\documentclass[12pt]{article}

\usepackage[utf8]{inputenc}

\usepackage{ graphicx, amsmath, amsthm, color, psfrag, amsfonts, natbib,
  url, amssymb, bm, thmtools, mathtools, fdsymbol }

\makeatletter
\renewcommand*{\@fnsymbol}[1]{\ensuremath{\ifcase#1\or *\or \dagger\or \ddagger\or
   \mathsection\or \mathparagraph\or \|\or **\or \dagger\dagger
   \or \ddagger\ddagger \else\@ctrerr\fi}}
\makeatother

\declaretheorem[ style = definition, numbered = yes ]{definition}
\declaretheorem[ style = proposition, numbered = yes ]{proposition}
\declaretheorem[ style = theorem, numbered = yes ]{theorem}
\declaretheorem[ style = lemma, numbered = yes ]{lemma}

\newcommand{\given}{\, | \,}

\newcounter{remark}
\newcommand{\remark}{%
  \stepcounter{remark}%
  \theremark}

\newcounter{example}
\newcommand{\example}{%
  \stepcounter{example}%
  \theexample}

\newcounter{algorithm}

\title{\mbox{} \vspace*{-1.0in} \\ \textbf{A Simple Necessary Condition \\ For Independence of Real-Valued \\ Random Variables}}

\author{\Large \textbf{David Draper\footnote{\textit{Address for correspondence}: David Draper, Department of Statistics, Baskin School of Engineering, University of California, 1156 High Street, Santa Cruz CA 95064 USA; email address \texttt{<draper@ucsc.edu>}. Additional email addresses: Erdong Guo$^\spadesuit$ \texttt{<eguo1@ucsc.edu>, Robert Lund$^\varheartsuit$ \texttt{<rolund@ucsc.edu>}, and Jon Woody$^\blacklozenge$ \texttt{<jwoody@math.msstate.edu>}}.} , Erdong Guo$^\spadesuit$, Robert Lund$^\varheartsuit$,} \vspace*{0.0125in} \\ \Large and \textbf{Jon Woody$^\blacklozenge$}}

\date{\textit{University of California, Santa Cruz (DD, EG, RL)} \\ \textit{Mississippi State University (JW)} \vspace*{0.2in} \\ 27 Nov 2021 \vspace*{-0.5in}}

\begin{document}

\maketitle

\begin{abstract}

\noindent
The standard method to check for independence of two real-valued random variables --- demonstrating that the bivariate joint distribution factors into the product of its marginals --- is both necessary and sufficient. Here we present a simple \textit{necessary} condition based on the \textit{support sets} of the random variables, which --- if not satisfied --- avoids the need to extract the marginals from the joint in demonstrating dependence. We review, in an accessible manner, the measure-theoretic, topological, and probabilistic details necessary to establish the background for the old and new ideas presented here. We prove our result in both the discrete case (where the basic ideas emerge in a simple setting), the continuous case (where serious complications emerge), and for general real-valued random variables, and we illustrate the use of our condition in three simple examples.

\bigskip

\noindent
\textbf{\textit{Keywords:}} Absolutely continuous CDF, \textit{amiable} PMF, Borel sets in $\mathbb{ R }^k$,  \textit{canonical} PDF version, closure of a set, continuous probability density function (PDF), cumulative distribution function (CDF), discrete probability mass function (PMF), IID (independent identically distributed) sampling, limit point of a set, Lebesgue measure in $\mathbb{ R }^k$, metric space, point of increase of a CDF,  probability space, SRS (simple random sampling), singular distribution, support set, topology of $\mathbb{ R }^k$, \textit{version} of a \textit{collection} of PDFs \textit{possessed} by an absolutely continuous CDF.

\end{abstract}

\section{Introduction} \label{s:introduction-1}

\textit{Independence} of two real-valued random variables $X$ and $Y$ is a bedrock idea in probability theory and statistical data science. The usual approach to checking for independence involves seeing whether the joint distribution factors into the product of the marginal distributions, which requires extracting the marginals from the joint. Here we offer a simple necessary condition for independence based instead on seeing whether the bivariate and marginal \textit{support sets} factor, which --- if they do not --- obviates the necessity to compute the marginal distributions.

The plan of the paper is as follows. In Section \ref{s:intuition-1} we present an intuitive summary of our main result, in the form of a relevant example. Section \ref{s:notation-definitions-preliminary-1} introduces notation, definitions, and preliminary results from the literature, including basic ideas from measure theory and topology. In Section \ref{s:discrete-case-1} we state and prove our result in the case of discrete real-valued random variables; Section \ref{s:continuous-case-1} provides parallel results in the continuous case. In Section \ref{s:general-case-1} we tell our story for general real-valued random variables. Section \ref{s:examples-1} offers several examples, and in Section \ref{s:discussion-1} we conclude the paper with a brief discussion.

Before beginning our main story, we note (prompted by Terenin (2021, personal communication)) that it's possible to prove our main result at an extremely high level of abstraction, but we've found that this obscures a number of details, relevant to practical data science, when working with random variables with values in $\mathbb{ R }^1$ and $\mathbb{ R }^2$; interested readers will find a more abstract proof sketch in the Appendix.

\section{An intuitive summary of our main result} \label{s:intuition-1}

We introduce our main finding intuitively in the context of the following example. \vspace*{-0.1in}

\begin{quote}

\textbf{Example \example.} Consider a darts player whose throws land on or inside a circle (except when they land outside the circle, in which case they're rejected with no penalty to the player); without loss of generality we can base our modeling of this situation on the unit circle in the real plane. Prior to the next throw (knowing nothing about previous throws, if any), we're uncertain about the Cartesian coordinates $( x, y )$ identifying where the dart will land, so as usual we can create a (continuous) bivariate random vector $\bm{ W } = ( X, Y )$ to quantify our uncertainty. Are the component random variables $X$ and $Y$ independent?

The standard method for answering this question involves (a) extracting the marginal probability density functions (PDFs) $f_X ( x )$ and $f_Y ( y )$ from the joint PDF $f_{ XY } ( x, y )$ and (b) seeing if the joint factors as the product of the marginals. This is not possible here, with the information given: at present we know nothing about the skill of the darts player, i.e., $f_{ XY } ( x, y )$ is not uniquely specified by problem context. However, if we add the assumption that \textit{all} points on or inside the circle are realizable places for the dart to land, it's immediate from the main result of this paper that $X$ and $Y$ are \textit{dependent}, even without any further knowledge about the joint PDF.

Our approach is based, not on the PDFs, but on the \textit{support sets} of $X$, $Y$, and $( X, Y )$: intuitively these are the \textit{nontrivial} subsets of the real line (for $X$ and $Y$) and the real plane (for $( X, Y )$), in the sense that (in the continuous case) they identify the values of the random variables with positive density. Denoting the relevant support sets here by $S_X$, $S_Y$, and $S_{ XY }$, respectively, our main result provides a \textit{necessary} condition for independence: if $X$ and $Y$ are independent, \textit{the support sets must factor}: $S_{ XY } = S_X \times S_Y$, in which $A \times B$ denotes the Cartesian product of the sets $A$ and $B$. In this example, evidently $S_{ XY }$ is all of the points in the real plane on or inside the unit circle and $S_X = S_Y = [ -1, 1 ]$, so that $S_X \times S_Y$ is the square that circumscribes the unit circle (a larger set than $S_{ XY }$); thus $X$ and $Y$ must be dependent. 

To see what's going on, consider two cases of this setting in which more is assumed about the bivariate PDF. The left panel of Figure \ref{f:intuitive-plots-1} presents both a contour plot of the Uniform PDF \{$f_{ XY } ( x, y ) = \frac{ 1 }{ \pi }$ for $( x, y )$ such that $( x^2 + y^2 ) \le 1$ and 0 otherwise\} (illustrating the repeated outcomes generated by a completely inept darts player) and a visualization of the support sets $S_{ XY }$ (the dark blue circle, together with the red circular boundary), $S_{ X }$ (the solid horizontal red line segment from $[ -1, -1 ]$ to $[ 1, -1 ]$), and $S_{ Y }$ (the solid vertical red line segment from $[ -1, -1 ]$ to $[ -1, 1 ]$); the light blue area represents the discrepancy between $S_{ XY }$ and $( S_X \times S_Y )$ in this case. The right panel of the figure gives a perspective plot of what may be termed a \textit{Roman Colosseum} PDF (this one has PDF \{$f_{ XY } ( x, y ) = \frac{ 2 }{ \pi } ( x^2 + y^2 )$ for $( x, y )$ such that $( x^2 + y^2 ) \le 1$ and 0 otherwise\}); this is more visually interesting than the Uniform PDF in the left panel and represents the repeated results of a highly skilled darts player who is rewarded for landing darts as close to the unit circle as possible.

\end{quote}

\begin{figure}[t!]

\centering

\caption{\textit{Left panel: contour plot of a Uniform PDF on the unit circle, with support sets indicated in red; right panel: perspective plot of a \textbf{Roman Colosseum} PDF (see text; higher PDF values in yellow)}.}

\label{f:intuitive-plots-1}

\begin{tabular}{cc}

\begin{tabular}{c}

\includegraphics[ scale = 0.4 ]{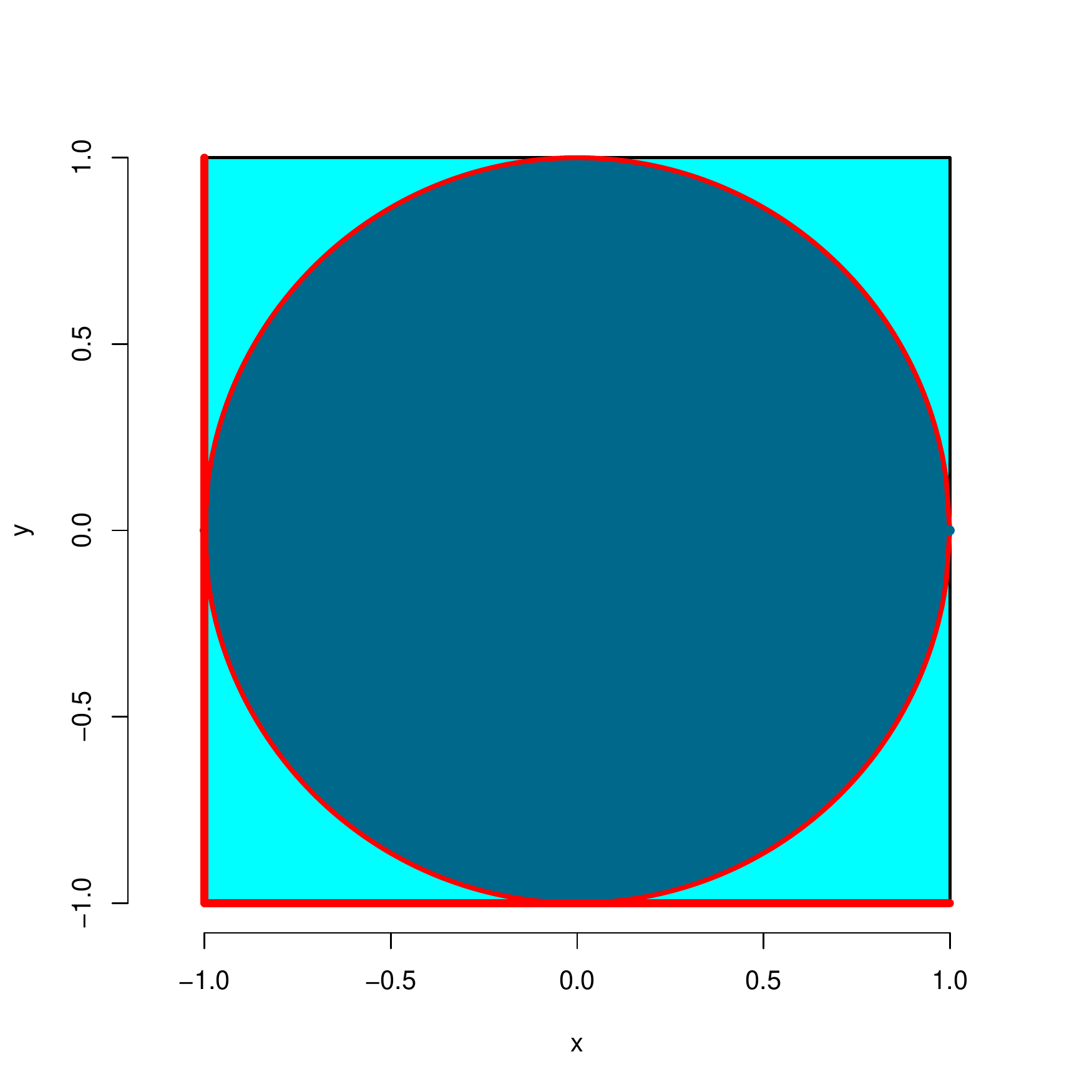} \\

\vspace*{0.3in}

\end{tabular}

&

\begin{tabular}{c}

\mbox{ } \vspace*{-0.3in} \\

\hspace*{-0.4in} \includegraphics[ scale = 0.3 ]{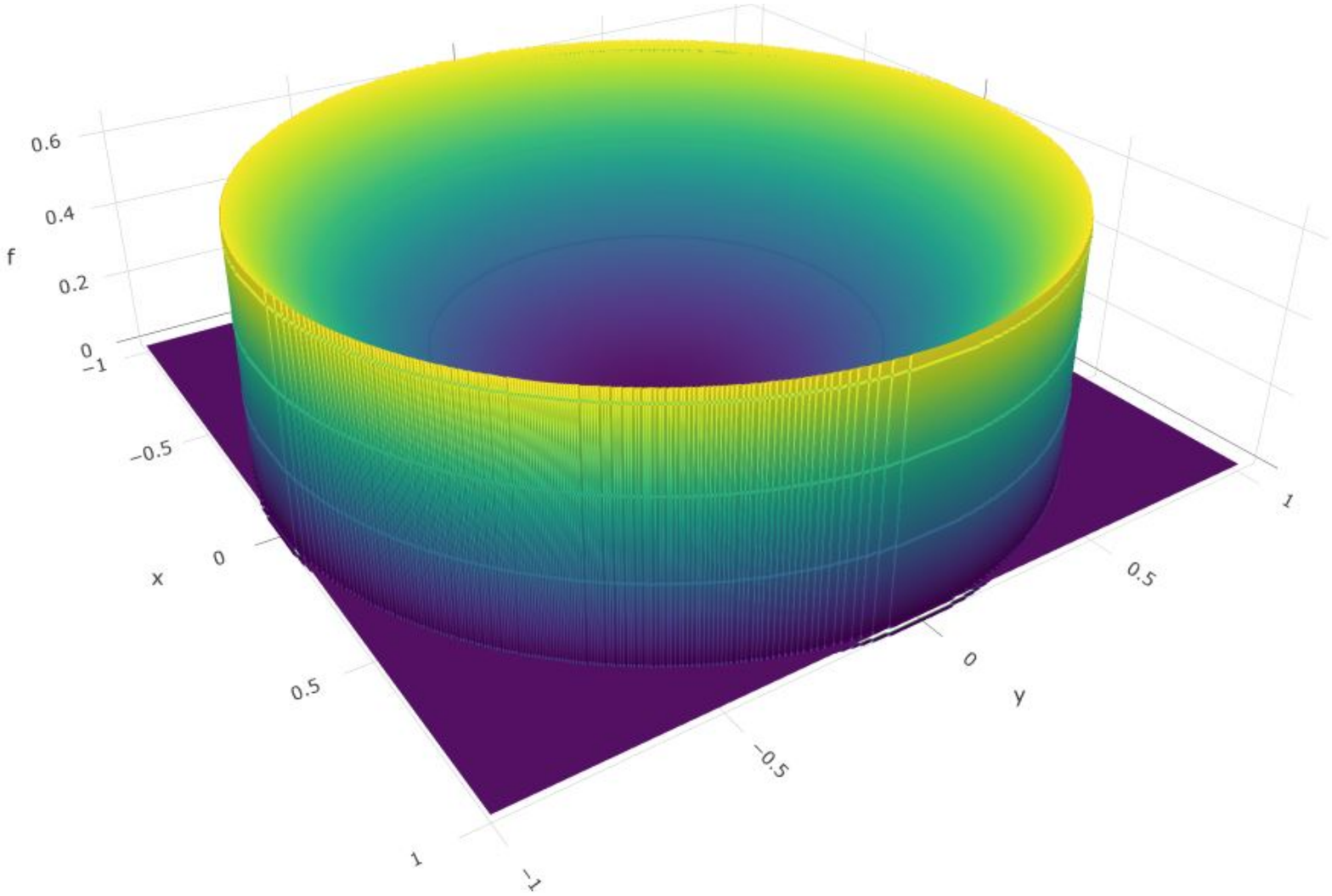}

\end{tabular}

\end{tabular}

\mbox{ } \\ \vspace*{-0.5in}

\end{figure}

Our main result is easy to state, but proving it in full generality turns out to involve engaging with several technical challenges: 
\begin{itemize}

\item

We need to be mindful of \textit{measure-theoretic} considerations, because (a) the general definition of independence for real-valued random variables involves measure theory and (b) PDFs of real-valued random variables are not uniquely defined (e.g., you can \textit{punch a hole} in the standard Normal PDF $\phi ( x ) = \frac{ 1 }{ \sqrt{ 2 \, \pi } } \exp \left( - \frac{ x^2 }{ 2 } \right)$ at any $x$ you like and replace its value there with, e.g., 0 without changing the probabilistic character of $\phi$); and 

\item

We need to be careful with our \textit{topological} details, because --- to work properly --- support sets need to be \textit{closed} (a key property of some, but not all, subsets of the real line and real plane). 

\end{itemize}

Along the way we'll encounter and cope with several nasty counterexamples to ordinary intuition, and we'll need to create several new definitions of properties shared by some, but not all, real-valued random variables, to make the theory match up with standard treatments of independence and support in (a) textbooks in probability and (b) papers in statistical data science.

\section{Notation, definitions, and preliminary results} \label{s:notation-definitions-preliminary-1}

\subsection{The probability context} \label{s:probability-context-1}

With $k$ as a finite positive integer, let $\mathcal{ S } = ( \Omega, \mathcal{ F }, P ) \triangleq ( \mathbb{ R }^k, \mathbb{ B }^k, P )$ be the \textit{probability space} (\cite{kolmogorov-1933}; e.g., \cite{breiman-1992}; throughout the paper the symbol $\triangleq$ means \textbf{\textit{is defined to be}}) in which 
\begin{itemize}

\item

the \textit{sample space} is $\Omega = \mathbb{ R }^k$; 

\item

the \textit{$\sigma$-algebra} is $\mathcal{ F } = \mathbb{ B }^k$, the Borel $\sigma$-field on $\mathbb{ R }^k$; and 

\item

$P \! : B \rightarrow [ 0, 1 ]$ is a \textit{probability measure}, in which $B$ is a set contained in $\mathbb{ B }^k$. 

\end{itemize}
This space is natural for working with real-valued random vectors $\bm{ X } = ( X_1, \, \dots, \, X_k )$ taking on values of the form $\bm{ x } = ( x_1, \, \dots, \, x_k )$, because (for $i = 1, \, \dots, \, k$) $X_i ( \bm{ x } ) = x_i$ simply picks out coordinate $i$ of the random vector. As special cases of this general notation in what follows, when $k = 1$ we work with single random variables with names such as $X$ and $Y$, and when $k = 2$ we use the notation $\bm{ W } = ( X, Y )$. As is customary, we write expressions such as $P ( a \le X \le b )$ as shorthand for $P \big[ x \! : a \le X ( x ) = x \le b \big]$. 

\begin{quote}

\textbf{Remark \remark.} Since we work here solely with the probability space $\mathcal{ S }$ defined above, to avoid repetition, in this paper (a) \textit{all mentions of the phrase \textbf{random variable} are abbreviations for the phrase \textbf{real-valued random variable}}, and (b) $k$ always represents the dimension of the real space under current consideration and is therefore always a finite positive integer.

\end{quote}

The Borel sets of particular interest to us are 
\begin{itemize}

\item

\textit{Neighborhoods} of a point in $\mathbb{ R }^k$: \textit{open balls} ($k$-spheres) $\mathcal{ B } ( \bm{ x }, \epsilon ) \triangleq \{ \bm{ x }^* \in \mathbb{ R }^k \! : || \bm{ x }^* - \bm{ x } || < \epsilon \}$ of radius $\epsilon > 0$ centered at $\bm{ x }$ (in which $\mathcal{ D } ( \bm{ x }^*, \bm{ x } ) \triangleq \sqrt{ \sum_{ i = 1 }^k ( x_i^* - x_i )^2 }$ is standard Euclidean distance); and

\item

\textit{Rectangles} in $\mathbb{ R }^k$ of the form $( x_1 - \epsilon_{ x_1 }, x_1 + \epsilon_{ x_1 } ) \times \dots \times ( x_k - \epsilon_{ x_k }, x_k + \epsilon_{ x_k } )$, for $\bm{ x } = ( x_1, \, \dots, \, x_k )$ and $\epsilon_i > 0$ for $( i = 1, \, \dots, \, k )$.

\end{itemize}
Let $\lambda_i$ denote Lebesgue measure on $\mathbb{ R }^i$ for $( i = 1, \, \dots, \, k )$ in what follows. Throughout the paper, we regard $\mathbb{ R }^k$ as (a) a metric space with the distance function $\mathcal{ D }$ mentioned above and (b) a topological space in which the basic open sets are open balls defined by the metric in (a). 

\subsection{Independence in $\mathbb{ R }^k$} \label{s:independence-1}

The general definition of independence of random variables defined on $\mathcal{ S }$  (see, e.g., \cite{shiryaev-1996}), when specialized to our notation with $k = 2$, is as follows: \vspace*{-0.35in}
\begin{quote}

\begin{definition} \label{d:independence-1} 

Random variables $( X, Y )$ are \textit{independent} iff for all Borel sets $\{ B_1, B_2 \}$ with $B_i \in \mathbb{ B }^1$ (for $i = 1, 2$)
\begin{equation} \label{e:independence-1}
P \big[ ( X \in B_1 ) \ \mathrm{and} \ ( Y \in B_2 ) \big] = P ( X \in B_1 ) \cdot P ( Y \in B_2 ) \, .
\end{equation}

\end{definition}

\textbf{Remark \remark.} As usual with definitions, this makes equation (\ref{e:independence-1}) a necessary and sufficient condition for independence; later, as noted in Sections \ref{s:introduction-1}--\ref{s:intuition-1}, as our main result we'll identify a condition that's necessary but not sufficent.

\end{quote}

An alternative approach to defining independence is as follows (proof omitted). \vspace*{-0.35in}
\begin{quote}

\begin{lemma} \label{l:alternative-independence-definition-1}

\textbf{(e.g., \cite{breiman-1992})} For all $( x, y ) \in \mathbb{ R }^2$, let $F_{ XY } ( x, y ) \triangleq P ( X \le x \ \mathrm{and} \ Y \le y )$ be the bivariate \textbf{cumulative distribution function} (CDF) of the random vector $\bm{ W } = ( X, Y )$, and denote by $F_X ( x ) \triangleq P ( X \le x )$ and $F_Y ( y ) \triangleq P ( Y \le y )$ the marginal CDFs for $X$ and $Y$, respectively, based on $F_{ XY } ( x, y )$. Then a necessary and sufficient condition for $X$ and $Y$ to be independent is that
\begin{equation} \label{e:independence-2}
F_{ XY } ( x, y ) = F_X ( x ) \cdot F_Y ( y ) \ \mathrm{for} \ \mathrm{all} \ ( x, y ) \in \mathbb{ R }^2 \, .
\end{equation}

\end{lemma}

\end{quote}

\subsection{Support sets in $\mathbb{ R }^k$} \label{s:support-sets-1}

Further simplification of \textbf{Definition \ref{d:independence-1}} is possible as a function of the \textit{univariate (marginal) support} of each of $X$ and $Y$ and the \textit{bivariate support} of $( X, Y )$. \vspace*{-0.1in}

\begin{quote}

\textbf{Remark \remark.} The basic idea of the support (set) of a random variable $X$ --- collecting together all values $x$ of $X$ that are nontrivial (i.e., that have positive probability, in the discrete case, or positive density, in the continuous case) --- is both natural and intuitive, but precisely defining this concept for general $X$ is a bit slippery. We follow \cite{billingsley-1995} in laying out the following definitions and lemma (proof omitted). The idea is (a) to define \textit{\textbf{a} (not the)} support (set), then (b) to define the \textit{minimal closed support set}, and then finally (c) to characterize the set in (b) in user-friendly ways.

\begin{definition} \label{d:support-1}

Given a general probability space $\mathcal{ S }^* = ( \Omega^*, \mathcal{ F }^*, P^* )$, \textit{\textbf{a} (not the) support (set)} of $P^*$ is any set $A \in \mathcal{ F }^*$ for which $P^* ( A ) = 1$. \vspace*{-0.1in}

\end{definition}

\textbf{Remark \remark.} By itself this definition is almost completely unhelpful; for example, the entire sample space $\Omega^*$ is \textit{\textbf{a}} support set of $P^*$. The next definition is where the key concept comes to life.

\begin{definition} \label{d:support-2}

The \textit{minimal closed support} of a probability measure $P$ on $\mathbb{ R }^k$ is a closed set $S_P$ such that 
\begin{equation} \label{d:minimal-support-1}
( S_P \subset C ) \ \textrm{for closed} \ C \ \ \ \textrm{iff} \ \ \ C \ \textrm{is \textbf{a} support set of} \  P \, ,
\end{equation}
in which the idea of a \textit{closed} set arises from the topology of $\mathbb{ R }^k$ when considered (as noted above) as a metric space with the Euclidean distance function $\mathcal{ D }$. When $P$ is a probability measure induced by a random variable $X$, we refer to $S_P$ as $S_X$, with similar notation and meaning for $S_Y$ and $S_{ XY }$.

\end{definition}

\begin{definition} \label{d:support-3}

Given a random variable $X$ with CDF $F_X ( x ) \triangleq P ( X \le x )$ (for all real $x$), a possible value $x^*$ of $X$ is called a \textit{point of increase} of $F_X$ iff for all $\epsilon > 0$ 
\begin{equation} \label{e:point-of-increase-1}
F_X ( x^* + \epsilon ) > F_X ( x^* - \epsilon ) \, .
\end{equation}

\end{definition}

\begin{lemma} \label{l:support-1}

\textbf{(\cite{billingsley-1995})} The set $S_P$ in \textbf{Definition \ref{d:support-2}} exists and is unique. Moreover, the minimal closed support set $S_P$ can be characterized in two equivalent ways in the setting of this paper:

\begin{itemize}

\item

for general $k$, $S_{ \bm{ X } } = \{ \bm{ x } \in \mathbb{ R }^k \! : P ( A ) > 0 \ \mathrm{for} \ \mathrm{all} \ \mathrm{open} \ \mathrm{neighborhoods} \ A \ \mathrm{of} \ \bm{ x } \}$; and

\item

for $k = 1$, and for a specific random variable $X$ with CDF $F_X ( x )$, $S_X$ is the set of all points of increase of $F_X$.

\end{itemize}

\end{lemma}

\begin{figure}[t!]

\centering

\caption{\textit{An approximation to the interesting portion of the Cantor CDF, for $x \in [ 0, 1 ]$, based on 1,023 partition sets along the $x$ axis}.}

\label{f:cantor-pdf}

\includegraphics[ scale = 0.6 ]{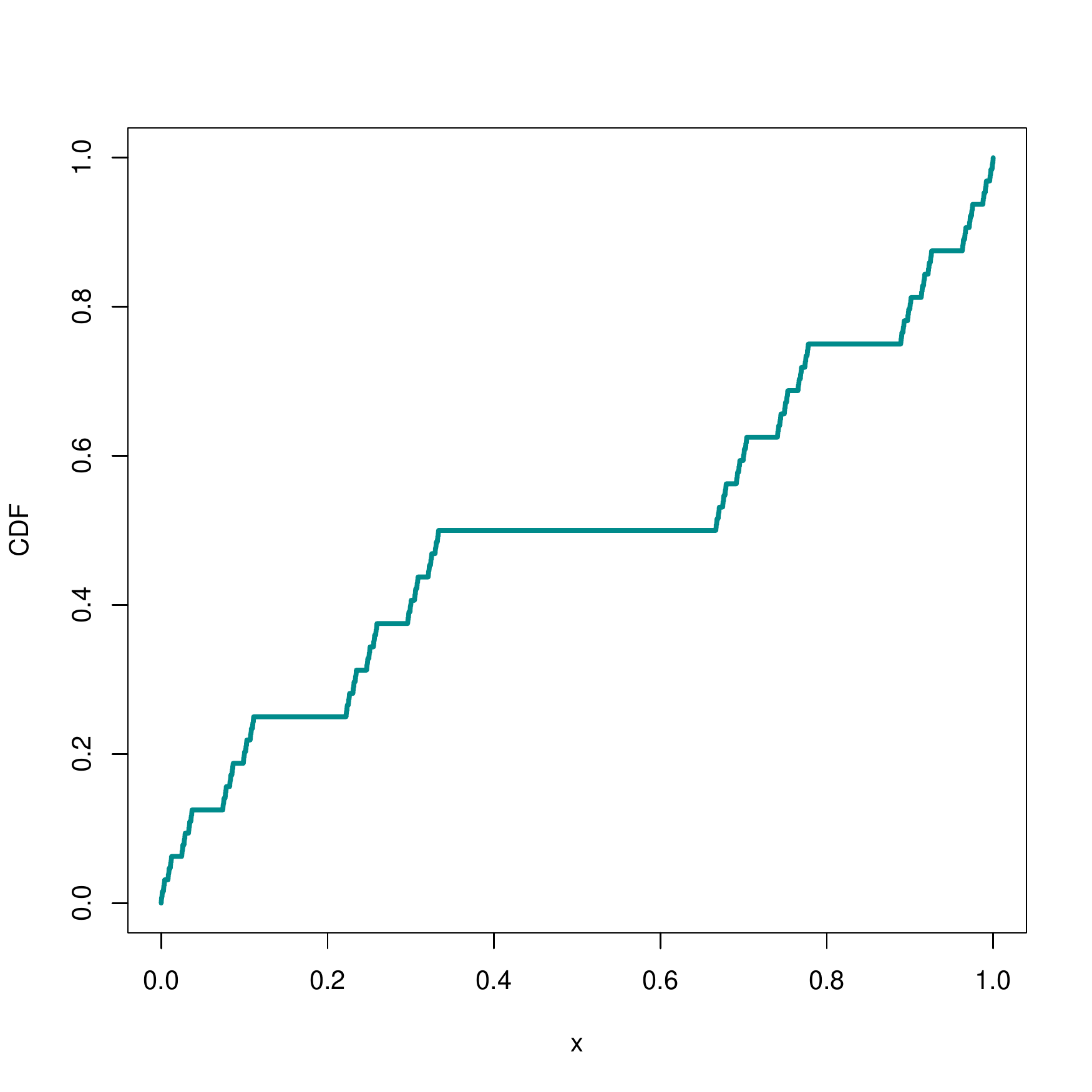}

\end{figure}

\textbf{Remark \remark.} This support machinery is general enough that it can even handle truly weird probability distributions on $\mathbb{ R }$. To set up a nasty example, first consider \textit{Lebesgue's Decomposition Theorem} (e.g., \cite{halmos-1974}) when applied to probability measures on the real line, which states that any CDF $F_X ( x )$ of a random variable $X$ can be expressed as the mixture
\begin{equation} \label{e:lebesgue-decomposition-1}
F_X ( x ) = \pi_1 \, F_X^D ( x ) \, + \, \pi_2 \, F_X^C ( x ) \, + \, \pi_3 \, F_X^S ( x ) \, ,
\end{equation} 
in which $( F_X^D, \, F_X^C, \, F_X^S )$ are CDFs of (discrete, (absolutely) continuous, singular) random variables (respectively) and where $0 \le \pi_i \le 1$ (for $i = 1, 2, 3$) with $\sum_{ i = 1 }^3 \pi_i = 1$. Weirdness can ensue whenever $\pi_3 > 0$ (note that a probability distribution $P$ on $\mathbb{ R }^1$ is \textit{singular} with respect to $\lambda_1$ iff $P$ concentrates all of its probability on a set of $\lambda_1$-measure 0).

\textbf{Example \example.} One of the most notorious probability distributions on $\mathbb{ R }$ is the CDF defined by the \textit{Cantor function} $\mathbb{ C } ( x )$ (e.g., \cite{dovgoshey-et-al-2006}), which has domain and range $[ 0, 1 ]$; this can be made into a CDF $F_{ \mathbb{ C } } ( x )$ on the entire real line by extending $\mathbb{ C } ( x )$ to include the values 0 and 1 on $( - \infty, 0 )$ and $( 1, \infty )$, respectively. The resulting nightmarish CDF (Figure \ref{f:cantor-pdf} illustrates an approximation for $x \in [ 0, 1 ]$) is everywhere continuous but has 0 derivative almost everywhere ($\lambda_1$); moreover, as noted, e.g., by \cite{shiryaev-1996}, letting $\mathcal{ N }$ be the set of points of increase of $F_{ \mathbb{ C } } ( x )$, one can show that $\lambda_1 ( \mathcal{ N } ) = 0$ but $P ( \mathcal{ N } ) = 1$ (in which $P$ is the probability measure on $\mathbb{ R }$ induced by $F_{ \mathbb{ C } }$). This demonstrates both (a) that $F_{ \mathbb{ C } }$ has $\pi_3 = 1$ in equation (\ref{e:lebesgue-decomposition-1}), i.e., that $F_{ \mathbb{ C } }$ defines an entirely singular distribution, and (b) that, even considering this weirdness, by \textbf{Definition \ref{d:support-2}} and \textbf{Lemma \ref{l:support-1}} the (minimal closed) support set for an $X$ with the Cantor CDF is perfectly well defined and equals the entire interval $[ 0, 1 ]$.

\textbf{Remark \remark.} Since the minimal closed support set always exists and is unique in the context of this paper, we simply call it \textit{the support set} or \textit{the support} in what follows.

\textbf{Remark \remark.} We conjecture that it's possible to prove the results examined here using the points-of-increase characterization of the support set in \textbf{Lemma \ref{l:support-1}}, by extending the idea of points-of-increase to $\mathbb{ R }^k$ for $k > 1$, for example using the following definition (which appears to be new to the literature): \vspace*{-0.25in}

\begin{quote}

\begin{definition} \label{d:new-point-of-increase-1}

\textbf{\textit{(new)}} Consider a random vector $\bm{ X } = ( X_1, \, \dots, \, X_k )$ with CDF $F_{ \bm{ X } }( \bm{ x } ) \triangleq P ( \bm{ X } \le \bm{ x } )$ (for all $\bm{ x } = ( x_1, \, \dots, \, x_k ) \in \mathbb{ R }^k$) and let $\epsilon$ be an arbitrary positive number. A possible value $\bm{ x^* } = ( x_1^*, \, \dots, \, x_k^* )$ of $\bm{ X }$ is called a \textit{point of increase} of $F_{ \bm{ X } }$ iff for all $( i = 1, \, \dots, \,  k )$
\begin{equation} \label{e:point-of-increase-2}
F_{ \bm{ X } } ( \, \dots, \, x_i^* + \epsilon, \, \dots \, ) > F_{ \bm{ X } } ( \, \dots, \, x_i^* - \epsilon, \, \dots \, ) \, ,
\end{equation}
in which the notation in equation (\ref{e:point-of-increase-2}) means \textit{hold all coordinates in $\bm{ x^* }$ constant except coordinate $i$ and compare the multivariate CDF at the two values $( x_i^* \pm \epsilon )$}.

\end{definition}

\end{quote}

\end{quote}

Instead of using \textbf{Definition \ref{d:new-point-of-increase-1}}, in the rest of the paper we use the neighborhood characterization in the \textbf{Lemma}, as follows. \vspace*{-0.35in}

\begin{quote}

\begin{definition} \label{d:support-4}

Consider a random vector $\bm{ W } = ( X, Y )$ with values $\bm{ w } = ( x, y ) \in \mathbb{ R }^2$. For any $\epsilon > 0$ let $\mathcal{ B } [ ( x, y ), \epsilon ]$ denote the open ball (circle) in $\mathbb{ R }^2$ of radius $\epsilon$ centered at $( x, y )$. Then the \textit{bivariate support} of $( X, Y )$ is the set
\begin{equation} \label{e:support-1}
S_{ XY } \triangleq \{ ( x, y ) \in \mathbb{ R }^2 \! : P \left\{ \mathcal{ B } [ ( x, y ), \epsilon ) ] \right\} > 0  \textrm{ for all } \epsilon > 0 \} \, ,
\end{equation}
and the \textit{marginal support} of $X$ is the set
\begin{equation} \label{e:support-2}
S_X \triangleq \{ x \in \mathbb{ R } \! : P \left[ ( x - \epsilon_x, x + \epsilon_x ) \right] > 0 \textrm{ for all } \epsilon_x > 0 \} \, ,
\end{equation}

with an analogous definition for $S_Y$, the marginal support of $Y$.

\end{definition}

\end{quote}

\section{The discrete case} \label{s:discrete-case-1}

\subsection{The support sets: the need for \textit{amiable} PMFs} \label{s:discrete-support-1}

When the set $S_{ XY }$ in equation (\ref{e:support-1}) is finite or at most countably infinite, it becomes meaningful to define the (joint) probability mass function (PMF) $p_{ XY } ( x, y )$ of $( X, Y )$ and the marginal PMFs $p_X ( x )$ and $p_Y ( y )$, respectively, in the usual way, as follows:
\vspace*{-0.35in}
\begin{quote}

\begin{definition} \label{d:discrete-case-1}

\textbf{\textit{(e.g., \cite{ash-2000})}} If the cardinality of $S_{ XY }$ is finite or countably infinite, $\bm{ W } = ( X, Y )$ is called a \textit{discrete random vector} with \textit{(joint) probability mass function} (PMF)
$p_{ XY } ( x, y ) \triangleq P ( X = x \textrm{ and } Y = y )$ (for all $( x, y ) \in \mathbb{ R }^2$) and with \textit{marginal} PMFs $p_X ( x ) \triangleq P ( X = x )$ (for all real $x$) and $p_Y ( y ) \triangleq P ( Y = y )$ (for all real $y$), respectively.

\end{definition}

\textbf{Remark \remark.} With reference to the \textit{Lebesgue Decomposition Theorem} on the CDF scale in equation (\ref{e:lebesgue-decomposition-1}), the discrete setting in this Section of the paper corresponds to $\bm{ \pi } \triangleq ( \pi_1, \pi_2, \pi_3 ) = ( 1, 0, 0 )$.

\end{quote}

Considering the marginal PMF of $X$ in this discrete case, it might be hoped that the marginal support set $S_X$ in equation (\ref{e:support-2}) would simply be 
\begin{equation} \label{little-s-not-big-S-1}
s_X \triangleq \{ x \in \mathbb{ R } \! : p_X ( x ) > 0 \}
\end{equation}
(note that $S_X$ and $s_X$ are not necessarily the same), but this is not correct in full generality, as the following unpleasant example (e.g., \cite{billingsley-1995}) shows.

\begin{quote}

\textbf{Example \example.} The rational numbers in $[ 0, 1 ]$ are countable, and may be enumerated using the diagonalization argument given by \cite{cantor-1891} specialized to the unit interval (there are many such enumerations, but they all lead to the same result in this context); call the resulting enumeration set $\mathcal{ E } \triangleq \{ x_1, x_2, \dots \}$, and construct the discrete random variable $X$ with PMF
\begin{equation} \label{e:unpleasant-example-1}
p_X ( x_n ) = \left\{ \begin{array}{ccc} 2^{ - n } & \textrm{if} & x_n \in ( s_X = \mathcal{ E } ) \\ 0 & & \textrm{otherwise} \end{array} \right\} \, .
\end{equation}
Now it turns out that the set of all rational numbers on $[ 0, 1 ]$ is not closed (in the metric/topological space of the real numbers, as specified in Section \ref{s:probability-context-1}), so the support of $X$ is \textit{not} the set $\mathcal{ E }$ of rationals on the unit interval (because all support sets are closed). It's natural to wonder if this can be remedied by working not with $s_X = \{ x \in \mathbb{ R } \! : p_X ( x ) > 0 \}$ but with its \textit{closure}. \vspace*{-0.25in}

\begin{quote}

\begin{definition} \label{d:closure-limit-points-1}

The \textit{closure} $\textit{cl} \, ( A )$ of a set $A \subseteq \mathbb{ R }$ is the union of $A$ with the set $L_A = \{ { x_L^A }_1, { x_L^A }_2 \, , \dots \}$ of all of its \textit{limit points}:
\begin{equation} \label{e:closure-limit-points-1}
\textit{cl} \, ( A ) = A \cup L_A \, .
\end{equation}Here ${ x_L^A }_i$ is a limit point of $A$ iff every neighborhood of ${ x_L^A }_i$ contains at least one point of $A$ different from ${ x_L^A }_i$.

\end{definition}

\end{quote}
Now, finally, since every real number is a limit of rational numbers, the support $S_X$ of the discrete random variable $X$ specified by the PMF in equation (\ref{e:unpleasant-example-1}) is the entire (closed) interval $[ 0, 1 ] = \textit{cl} \, ( s_X )$.

\textbf{Example \example}. A difficulty similar to that in \textbf{Example 3} arises with the slightly less unpleasant but still problematic PMF
\begin{equation} \label{e:unpleasant-example-2}
p_X ( x ) = \left\{ \begin{array}{ccc} x & \textrm{if} & x \in \left( s_X = \{ \frac{ 1 }{ 2 }, \frac{ 1 }{ 4 }, \frac{ 1 }{ 8 }, \, \dots \} \right) \\ 0 & & \textrm{otherwise} \end{array} \right\} \, .
\end{equation}
As in \textbf{Example 3}, $s_X$ is not closed and therefore cannot be the support of this random variable; note that here $s_X$ has the single limit point $\{ 0 \}$. From equation (\ref{e:support-2}) in \textbf{Definition \ref{d:support-4}} it's again evident that $S_X = \textit{cl} \, ( s_X )$.

\end{quote}
It's now reasonable to conjecture that working with $\textit{cl} \, ( s_X )$ instead of $s_X$ solves the problem identified by \textbf{Examples 3--4} in general, not just in those examples. The following result demonstrates that this is indeed true, in the case of a single discrete random variable $X$ (the proof is similar in $\mathbb{ R }^2$ for $( X, Y )$). \vspace*{-0.35in}

\begin{quote} 

\begin{lemma} \label{l:S-X-works-not-s-X-1}

\textbf{\textit{(new)}} Let $X$ be a discrete random variable with PMF $p_X ( x )$, and define $s_X \triangleq \{ x \in \mathbb{ R } \! : p_X ( x ) > 0 \}$. In this setting the general definition of support in equation (\ref{e:support-2}) becomes 
\begin{equation} \label{e:pmf-1}
S_X = \textit{cl} \, ( s_X ) \, ,
\end{equation}
with an analogous expression for $S_{ XY }$ when $X$ and $Y$ are both discrete.

\end{lemma}

\begin{proof}

We give details for $k = 1$. There are two cases to consider: 

\begin{itemize}

\item

In case (i), exemplified (for instance) by a Poisson$( \eta )$ PMF with any $\eta > 0$, $L_{ s_X } = \varnothing$, because if you choose any $x^*$ such that $p_X ( x^* ) = 0$, for small enough $\epsilon$ the neighborhood around $x^*$ with radius $\epsilon$ will have 0 probability under $p_X$. In this case, by equation (\ref{e:support-2}) in \textbf{Definition \ref{d:support-4}}, $S_X = s_X = ( s_X \cup \varnothing ) = ( s_X \cup L_{ s_X } ) = \textit{cl} \, ( s_X )$, yielding equation (\ref{e:pmf-1}); and

\item

In case (ii), illustrated by the unpleasant \textbf{Examples 3--4} above, $S_X \ne s_X$ because $L_{ s_X } \ne \varnothing$; but in this setting if we construct $S_X = \textit{cl} \, ( s_X ) = s_X \cup L_{ s_X }$, every neighborhood of every point in $S_X$ will have positive probability, again yielding (\ref{e:pmf-1}).

\end{itemize}

\end{proof}

\end{quote}

To rule out unpleasant PMFs such as those in \textbf{Examples 3--4}, which have no useful place in practical data science, in what follows we restrict attention solely to the discrete distributions in case (i) in the proof of \textbf{Lemma \ref{l:S-X-works-not-s-X-1}}, by making the following definition. \vspace*{-0.35in}

\begin{quote}

\begin{definition} \label{d:practical-support-1}

\textbf{\textit{(new)}} A discrete random variable $X$ with PMF $p_X ( x )$ is said to be \textbf{\textit{amiable}} iff $S_X = s_X = \{ x \in \mathbb{ R } \! : p_X ( x ) > 0 \}$, with an analogous definition for $S_{ XY }$ when $X$ is part of a bivariate random vector; note that amiability is equivalent to the conditions (i) that $L_{ s_X } = \varnothing$ and (ii) that $s_X$ is closed.

\end{definition}

\textbf{Remark \remark.} All well-regarded introductory probability textbooks (e.g., \cite{degroot-schervish-2012}) are written as if all discrete random variables are amiable in the sense of \textbf{Definition \ref{d:practical-support-1}}, and with good reason: \textit{all} of the standard PMFs in everyday probability and data science (e.g., Bernoulli, Beta-Binomial, Binomial, Hypergeometric, Negative Binomial (including Geometric), Poisson, and (discrete) Uniform) are amiable, because for each of them the corresponding set $s_X$ has no limit points. 

\end{quote}

\subsection{The probabilities under independence in the discrete setting} 
\label{s:discrete-probabilities-independence-1}

Consider what happens to the general definition of independence of two random variables $X$ and $Y$ in the discrete case. \textbf{Definition \ref{d:independence-1}} says that $P \big[ ( X \in B_1 ) \ \mathrm{and} \ ( Y \in B_2 ) \big] = P ( X \in B_1 ) \cdot P ( Y \in B_2 )$ for \textit{all} Borel sets $B_1$ and $B_2$ in $\mathbb{ B }^1$. So consider \textit{singleton} sets of the form $( X = x )$ and $( Y = y )$; all such sets are Borel. Thus \textbf{Definition \ref{d:independence-1}} in the discrete case specializes to the following: if discrete $X$ and $Y$ are independent, then
\begin{equation} \label{e:independence-3}
p_{ XY } ( x, y ) = p_X ( x ) \cdot p_Y ( y ) \textrm{ for all } ( x, y ) \textrm{ in } \mathbb{ R }^2 \, .
\end{equation}
The converse is also true (proof omitted).

\subsection{Our result in the discrete case} \label{s:discrete-result-1}

We can now state and prove our main result in the discrete case, which illustrates the basic ideas of the proof in the general setting of Section \ref{s:general-case-1}. \vspace*{-0.35in}

\begin{quote}

\begin{proposition} \label{p:discrete-proposition-statement-1}

\textbf{\textit{(new)}} Let $X$ and $Y$ be amiable discrete random variables with support sets and PMFs $\{ S_X, p_X ( x ) \}$ and $\{ S_Y, p_Y ( y ) \}$, respectively, and joint support set and PMF $\{ S_{ XY }, p_{ XY } ( x, y ) \}$. If $X$ and $Y$ are independent, then
\begin{equation} \label{e:discrete-support-1}
S_{ XY } = S_X \times S_Y \, .
\end{equation}
\vspace*{-0.45in}

\end{proposition}

\begin{proof}

Under amiability and independence, from equation (\ref{e:pmf-1})  the bivariate support set takes the form
\begin{equation} \label{s:bivariate-support-1}
S_{ XY } = \{ ( x, y ) \in \mathbb{ R }^2 \! : p_{ XY } ( x, y ) > 0 \} = \{ ( x, y ) \in \mathbb{ R }^2 \! : p_X ( x ) \cdot p_Y ( y ) > 0 \} \, .
\end{equation}
Now we use a basic fact about real numbers: 

\begin{quote}

\textbf{Fact $\bm{ ( * ) }$} If real $a$ and $b$ are both non-negative and their product is positive, they must both be positive. 

\end{quote}
This implies, using \textbf{Fact $\bm{ ( * ) }$} and the definition of a Cartesian product, that
\begin{eqnarray} \label{s:bivariate-support-2}
S_{ XY } & = & \{ ( x, y ) \in \mathbb{ R }^2 \! : p_X ( x ) > 0 \textrm{ and } p_Y ( y ) > 0 \} \nonumber \\
& = & \{ x \in \mathbb{ R } \! : p_X ( x ) > 0 \} \times \{ y \in \mathbb{ R } \! : p_Y ( y ) > 0 \} = S_X \times S_Y \, ,
\end{eqnarray}
as desired.

\end{proof}

\textbf{Remark \remark.} The contrapositive form of \textbf{Proposition \ref{p:discrete-proposition-statement-1}} establishes a \textit{necessary} condition for independence with amiable discrete random variables: for example, to demonstrate that amiable discrete $X$ and $Y$ are \textit{dependent}, all you have to do is to show that their bivariate support set doesn't factor, which is typically easier than (a) extracting their marginal PMFs and (b) showing that the bivariate PMF is different from the product of the marginals.

\textbf{Remark \remark.} Our necessary condition for independence is of course far from sufficient; it's easy to construct settings in which the support sets factor but the PMFs do not (e.g., see \textbf{Examples 7--8} in Section \ref{s:examples-1} below).

\textbf{Remark \remark. Proposition \ref{p:discrete-proposition-statement-1}} has an equivalent formulation in terms of \textit{conditional} PMFs and support sets, as follows. For a fixed $x \in \mathbb{ R }$ such that $p_X ( x ) > 0$, and continuing the case in which discrete $X$ and $Y$ are each amiable, define the \textit{conditional support of $Y$ given $( X = x )$} to be
\begin{equation} \label{e:conditional-support-1}
S_{ ( Y \given X = x ) } \triangleq \{ y \in \mathbb{ R } \! : p_{ \, Y \given X } ( y \given x ) > 0 \}, 
\end{equation}
in which $p_{ \, Y \given X } ( y \given x ) = P ( Y = y \given X = x )$ is the conditional PMF of $Y$ given $( X = x )$, and let $S_{ ( X \given Y = y ) }$ be defined analogously. Then (the trivial proof is omitted)

\begin{proposition} \label{p:discrete-proposition-2}

\textbf{\textit{(new)}} Let $X$ and $Y$ be amiable discrete random variables with marginal and conditional support sets $\{ S_X , S_{ ( X \given Y = y ) } \}$ and $\{ S_Y, S_{ ( Y \given X = x ) } \}$, respectively. If $X$ and $Y$ are independent, then
\begin{equation} \label{e:discrete-support-2}
\mathrm{for \ all} \ y \in \mathbb{ R }, \ S_{ ( X \given Y = y ) } = S_X \ \ \ \mathrm{and} \ \ \ \mathrm{for \ all} \ x \in \mathbb{ R }, \ S_{ ( Y \given X = x ) } = S_Y  \, .
\end{equation}

\end{proposition}

\end{quote}

\section{The continuous case} \label{s:continuous-case-1}

\subsection{The support sets: the need for \textit{canonical} PDFs} \label{s:continuous-support-1}

When the set $S_{ XY }$ in equation (\ref{e:support-1})  is uncountably infinite, it becomes meaningful to define a (joint) probability density function (PDF) for $( X, Y )$ and marginal PDFs $f_X ( x )$ and $f_Y ( y )$, respectively, in the usual way, as follows.
\vspace*{-0.35in}
\begin{quote}

\begin{definition} \label{d:continuous-1}

\textbf{\textit{(e.g., \cite{feller-1970})}} Let $X$ and $Y$ be random variables with joint CDF $F_{ XY } ( x, y )$ and marginal CDFs $F_X ( x )$ and $F_Y ( y )$, and consider the \textit{continuous} case in which (the bivariate probability measure induced by) $F_{ XY } ( x, y )$ is \textit{absolutely continuous} with respect to $\lambda_2$. This means (i) that \textit{\textbf{a} (not the)} non-negative bivariate \textit{probability density function} (PDF) $f_{ XY } ( x, y )$ exists with the property that for almost all (Lebesgue) $( x, y ) \in \mathbb{R }^2$ 
\begin{equation} \label{e:continuous-case-1}
F_{ XY } ( x, y ) = \int_{ - \infty }^x \int_{ - \infty }^y f_{ XY } ( s, t ) \, dt \, ds \ \ \ \textrm{and} \ \ \ \frac{ \partial^2 }{ \partial x \, \partial y } F_{ XY } ( x, y ) = f_{ XY } ( x, y ) \, ,
\end{equation}
and (ii) that \textit{\textbf{a} (not the)} non-negative \textit{marginal PDF} $f_X ( x )$ for $X$ and \textit{\textbf{a} (not the)} non-negative marginal PDF $f_Y ( y )$ for $Y$ exist such that for almost all (Lebesgue) $x$ and $y \in \mathbb{ R }^1$ 
\begin{equation} \label{e:continuous-case-2}
F_X ( x ) = \int_{ - \infty }^x  f_X ( s ) \, ds \, \ \ \ \textrm{and} \ \ \ \frac{ \partial }{ \partial x } F_X ( x ) = f_X ( x ) \, .
\end{equation}
and
\begin{equation} \label{e:continuous-case-3}
F_Y ( y ) = \int_{ - \infty }^y  f_Y ( t ) \, dt \, \ \ \ \textrm{and} \ \ \ \frac{ \partial }{ \partial y } F_Y ( y ) = f_Y ( y ) \, .
\end{equation}

\end{definition}

\textbf{Remark \remark.} Unfortunately, as is well known, this means that, under absolute continuity of $F_{ XY }$, the $f_{ XY } ( s, t )$ in (\ref{e:continuous-case-1}) and the $f_X ( s )$ and $f_Y ( t )$ in (\ref{e:continuous-case-2}--\ref{e:continuous-case-3}) are unique only up to sets of Lebesgue measure 0; thus

\begin{itemize}

\item

there are infinitely many possible densities associated with each of $\{ F_{ XY }, F_X,$ $F_Y \}$, and

\item

all that we can say for sure about the differentiability of each of the CDFs $\{ F_{ XY }, F_X,$ $F_Y \}$ is that (a) $F_{ XY }$ is twice differentiable (once in $x$, once in $y$) almost everywhere ($\lambda_2$) and (b) each of $F_X$ and $F_Y$ are (once) differentiable almost everywhere ($\lambda_1$).

\end{itemize}

This prompts the following \textbf{Definition}, which is new only in its proposed terminology.

\begin{definition} \label{d:version-collection-1}

\textbf{\textit{(new)}} We refer to any $f_{ XY }$ satisfying equation (\ref{e:continuous-case-1}) as a \textit{version} $f_{ XY }^{ \, * }$ of the \textit{collection} $\{ f_{ XY }^{ \, t } ( x, y ), t \in \mathbb{ R } \}$ of PDFs \textit{possessed by} an absolutely continuous CDF $F_{ XY }$, with similar terminology and notation for $f_X^{ \, u }$ and $f_Y^{ \, v }$.

\end{definition}

\textbf{Remark \remark.} With reference to the \textit{Lebesgue Decomposition Theorem} on the CDF scale in equation (\ref{e:lebesgue-decomposition-1}), the continuous setting in this Section of the paper corresponds to $\bm{ \pi } = ( 0, 1, 0 )$.

\textbf{Remark \remark.} Note that, in the absolutely continuous case, all singletons $\{ \hat{ x } \} \in \mathbb{ R }$ and $\{ \bm{ \hat{ x } } \} \in \mathbb{ R }^2$ have probability 0 under \textit{all} versions of the relevant densities $\{ f_{ XY }^{ \, * }, f_X^{ \, * }, f_Y^{ \, * } \}$.

\end{quote}

Consider first the marginal distribution for $X$ defined by $F_{ XY }$. As in the discrete case, for any version $f_X^{ \, * }$, the set $s_X^{ \, * } = \{ x \in \mathbb{ R } \! : f_X^{ \, * } ( x ) > 0 \}$ is of considerable interest, as is its closure $\left( s_X^{ \, * } \right)_{ \textit{cl} \, } \triangleq \textit{cl} \, ( s_X^{ \, * } )$. It might be hoped that all PDF versions would share the same $\textit{cl} \, ( s_X^{ \, * } )$, but this is not true, as the following example shows. \vspace*{-0.15in}

\begin{quote}

\textbf{Example \example.} Consider two versions of the familiar continuous Uniform distribution on the unit interval:
\begin{equation} \label{e:uniform-versions-1}
f_X^{ \, 1 } ( x ) = \left\{ \begin{array}{ccc} 1 & \textrm{if} & x \in ( 0, 1 ) \\ 0 & & \textrm{otherwise} \end{array} \right\} \ \ \ \textrm{and} \ \ \ f_X^{ \, 2 } ( x ) = \left\{ \begin{array}{ccc} 1 & \textrm{if} & x \in \big\{ ( 0, 1 ) \, \cup \, \{ \hat{ x } \} \big\} \\ 0 & & \textrm{otherwise} \end{array} \right\} \, ,
\end{equation}
in which $\hat{ x }$ is any singleton not contained in $( 0, 1 )$. It may appear at first that these two versions share the same set $\textit{cl} \, ( s_X^{ \, i } )$; however, singletons on the real number line are closed sets (in the topology relevant to this paper; see Section \ref{s:probability-context-1}), so $\textit{cl} \, ( s_X^{ \, 1 } ) = [ 0, 1 ] \ne \textit{cl} \, ( s_X^{ \, 2 } ) = [ 0, 1 ] \cup \{ \hat{ x } \}$. This slightly unpleasant example can be extended further to a much nastier version $f_X^{ \, 3 }$ for which $\textit{cl} \, ( s_X^{ \, 3 } )$ is the entire real line (by putting point masses of height 1 at all of the rational numbers in $\mathbb{ R }$ outside $( 0, 1 )$), even though all of the probability associated with all three versions is concentrated solely on $( 0, 1 )$.

\textbf{Remark \remark.} Excellent introductory (non-measure-theoretic) textbooks (e.g., \cite{degroot-schervish-2012}, p.~101) can be found which state (edited to match our notation) that ``If $X$ has a continuous distribution,~...~the closure of the set $\{ x \! : f_X ( x ) > 0 \}$ is called the \textit{support} of (the distribution of) $X$.'' As \textbf{Example 5} shows, however, this cannot be the whole story in full generality, because different versions $f_X^{ \, * }$ in the PDF collection possessed by an absolutely continuous CDF $F_X$ can have different sets of the form $\textit{cl} \, (  \{ x \! : f_X^{ \, * } ( x ) > 0 \} )$.

\end{quote}

In parallel with the discrete case, we wish to rule out unpleasant PDFs such as $f_X^{ \, 2 }$ and $f_X^{ \, 3 }$ in \textbf{Example 5}, which have no useful place in practical data science. Given a particular CDF of interest, the best way to remedy this problem turns out to be to \textit{construct} a special PDF version from the CDF, as in the following definition. \vspace*{-0.35in}

\begin{quote}

\begin{definition} \label{d:canonical-pdf-1}

Let $F_X$ be an absolutely continuous CDF on $\mathbb{ R }$ with collection $\{ f_X^{ \, t } ( x ), t \in \mathbb{ R } \}$ of PDF versions. For any $x^* \in \mathbb{ R }$ there are two possibilities: either $F_X$ is differentiable at $x^*$ or it's not (recall from \textbf{Remark 13} that all we know for sure is that $F_X$ is differentiable almost everywhere ($\lambda_1$)). Define the \textit{canonical} version $f_X^{ \, C }$ as follows:
\begin{equation} \label{e:canonical-pdf-1}
f_X^{ \, C } ( x^* ) = \left\{ \begin{array}{ccc} \left[ \frac{ \partial }{ \partial \, x } \, F_X ( x ) \right]_{ x = x^* } & \textrm{if} & F_X \textrm{ is differentiable at } x^* \\ 0 & & \textrm{otherwise} \end{array} \right\} \, ,
\end{equation}
with an analogous expression for $f_Y^{ \, C } ( y^* )$. The corresponding definition for the joint CDF $F_{ XY }$ is
\begin{equation} \label{e:canonical-pdf-2}
f_{ XY }^{ \, C } ( x^*, y^* ) = \left\{ \begin{array}{ccc} \left[ \frac{ \partial^2 }{ \partial x \, \partial y } \, F_{ XY } ( x, y ) \right]_{ ( x, y ) = ( x^*, y^* ) } & \textrm{if} & \left[ \begin{array}{c} F_{ XY } \textrm{ is twice} \\ \textrm{differentiable (once in } x, \\ \textrm{once in } y ) \textrm{ at } ( x^*, y^* ) \end{array} \right] \\ 0 & & \textrm{otherwise} \end{array} \right\} \, .
\end{equation}

\end{definition}

\textbf{Remark \remark.} Since the PDFs $\{ f_X^{ \, C } ( x^* ), f_Y^{ \, C } ( y^* ), f_{ XY }^{ \, C } ( x^*, y^* ) \}$ have been defined in equations (\ref{e:canonical-pdf-1}-\ref{e:canonical-pdf-2}) in a pointwise manner with no ambiguity at any point, it's clear that the canonical PDFs defined in this way always exist and are unique.

\textbf{Remark \remark.} Note that the unpleasant PDF versions $f_X^{ \, 2 }$ and $f_X^{ \, 3 }$ in \textbf{Example 5} lose their power to confound us if we restrict attention to the canonical PDF version of the Uniform$( 0, 1 )$ distribution: all three of the PDF versions $\{ f_X^{ \, i } \}$ (for $i = 1, 2, 3$) arise from the same CDF, which (as usual with $X \sim$ Uniform$( 0, 1 )$) is 
\begin{equation} \label{e:uniform-cdf-1}
F_X ( x ) = \left\{ \begin{array}{ccc} 0 & \textrm{if} & x \le 0 \\ x & & 0 \le x \le 1 \\ 1 & & x \ge 1 \end{array} \right\} \, ,
\end{equation}
but we're free to choose the canonical PDF version if we wish and ignore all other versions as unhelpful in day-to-day statistical modeling.

\textbf{Example \example.} Consider the following examples of the application of \textbf{Definition 12}.

\begin{itemize}

\item

The standard Normal CDF $\Phi ( x )$ is differentiable for all $x \in \mathbb{ R }$, so the canonical PDF corresponding to the standard Normal CDF is the usual textbook Gaussian PDF $\phi ( x )$. If we try to work instead with a Normal PDF that equals $\phi ( x )$ except at a finite or countably infinite set of singletons, the result cannot be canonical, because $\Phi ( x )$ is everywhere differentiable. The same remarks (of course) apply to the entire $N ( \mu, \sigma^2 )$ family of CDFs and PDFs.

\item

One parameterization of the usual family of Exponential$( \eta )$ distributions (indexed by $\eta > 0$) has a family of CDFs given by the expression
\begin{equation} \label{e:exponential-cdf-1}
F_X ( x ) = \left\{ \begin{array}{ccc} 0 & \textrm{for} & x \le 0 \\ 1 - e^{ - \eta \, x } & \textrm{for} & x \ge 0 \end{array} \right\} \, .
\end{equation}
This CDF is differentiable everywhere except at the point $x = \{ 0 \}$; the resulting canonical PDF family then becomes
\begin{equation} \label{e:exponential-pdf-1}
f_X^{ \, C } ( x ) = \left\{ \begin{array}{ccc} \eta \, e^{ - \eta \, x } & \textrm{for} & x > 0 \\ 0 & & \textrm{else} \end{array} \right\} \, .
\end{equation}
Note for this example that $s_X^{ \, C } = \{ x \in \mathbb{ R } \! : f_X^{ \, C } ( x ) > 0 \} = ( 0, \infty )$, an open set whose closure is $\textit{cl} \, ( s_X^{ \, C } ) = [ 0, \infty )$. Many (but not all) good introductory probability texts define \text{the} Exponential$( \eta )$ PDF as in equation (\ref{e:exponential-pdf-1}); others define the positive-PDF region to be $[ 0, \infty )$, for which the closure is still $[ 0, \infty )$.

\end{itemize}

\textbf{Remark \remark.} In typical statistical data-science modeling with an unknown quantity about which (from problem context) our uncertainty is continuous, we're accustomed to concentrating on the \underline{P}DF except when we need to compute (e.g.) tail areas, when (of course) we need to think about the \underline{C}DF. This can put us into a mindset in which we \textit{start} with the PDF and ask ``What CDF corresponds to this PDF?'' It's crucial to our definition of canonical PDFs that we think about the CDF--PDF relationship \textit{in the other direction}: we start with the \underline{C}DF and ask ``What version of the collection of \underline{P}DFs possessed by this \underline{C}DF is the most useful?'' As an extreme instance of what happens when we try to create a counterexample to our canonical--PDF definition, consider the following PDF version of the Uniform$( 0, 1 )$ CDF, which follows on from \textbf{Example 5}:
\begin{equation} \label{e:uniform-versions-2}
f_X^{ \, 4 } ( x ) = \left\{ \begin{array}{ccc} 1 & \textrm{if} & x \in ( 0, 1 ) \textrm{ and } x \textrm{ is irrational} \\ 0 & & \textrm{otherwise} \end{array} \right\} \, .
\end{equation}
This deeply unpleasant version is obtained by punching a hole in version 1 of the Uniform PDF in equation (\ref{e:uniform-versions-1}) at every rational number in the unit interval (a countable collection of singletons) and replacing the value of the PDF version there with 0. However, \textit{the CDF possessing this PDF version is still the usual Uniform CDF given by equation (\ref{e:uniform-cdf-1})}, so $f_X^{ \, 4 }$ cannot be canonical, and in fact is only useful from a statistically practical point of view as a curiosity that can be safely ignored.

\textbf{Remark \remark.} In the rest of this section \textit{we work only with the canonical versions of all PDF collections}.

\end{quote}

We're now ready to identify the manner in which the general support set definition in equation (\ref{e:support-2}) specializes in the continuous case. \vspace*{-0.35in}

\begin{quote} 

\begin{lemma} \label{l:S-X-works-not-s-X}

\textbf{\textit{(new)}} Let $F_X ( x )$ be an absolutely continuous CDF with canonical PDF $f_X^{ \, C }$ defining the sets $s_X^{ \, C } \triangleq \{ x \in \mathbb{ R } \! : f_X^{ \, C } ( x ) > 0 \}$ and $\left( s_X^{ \, C } \right)_{ \textit{cl} \, }\triangleq \textit{cl} \, ( s_X^{ \, C } )$. In this setting the general definition of support in equation (\ref{e:support-2}) becomes 
\begin{equation} \label{e:pmf-2}
S_X \triangleq \{ x \in \mathbb{ R } \! : P \left[ ( x - \epsilon_x, x + \epsilon_x ) \right] > 0  \textrm{ for all } \epsilon_x > 0 \} = \textit{cl} \, ( s_X^{ \, C } ) \, ,
\end{equation}
with analogous expressions for $S_Y$ and $S_{ XY }$ when $X$ and $Y$ are both continuous.

\end{lemma}

\begin{proof}

One of the simplest ways to show that two sets are equal is to demonstrate that each is a subset of the other, so the proof below is in two parts: (a) $S_X \subseteq \textit{cl} \, ( s_X^{ \, C } )$ and (b) $\textit{cl} \, ( s_X^{ \, C } ) \subseteq S_X$.

\begin{itemize}

\item[(a)]

$[ S_X \subseteq \textit{cl} \, ( s_X^{ \, C } ) ]$: Choose an arbitrary $x_0 \in S_X$; then for all $\epsilon_{ x_0 } > 0$ we have that
\begin{equation} \label{e:lemma-4-1}
P [ ( x_0 - \epsilon_{ x_0 }, x_0 + \epsilon_{ x_0 } ) ] > 0 \, .
\end{equation}
The only way that the probability on the left side of equation (\ref{e:lemma-4-1}) can be positive for all $\epsilon_{ x_0 } > 0$ is either (i) $f_X^{ \, C } ( x_0 ) > 0$ or (ii) $x_0$ is a limit point of $s_X^{ \, C }$ (because, if neither of these things were true --- i.e., if $f_X^{ \, C } ( x_0 ) = 0$ and there does not exist a point $x_{ 00 }$ (different from $x_0$) in $s_X^{ \, C }$ with $f_X^{ \, C } ( x_{ 00 } ) > 0$ --- it would then follow that $P ( x_0 - \epsilon_{ x_0 }, x_0 + \epsilon_{ x_0 } ) = 0$). Thus $x_0 \in \textit{cl} \, ( s_X^{ \, C } )$.

\item[(b)]

$[ \textit{cl} \, ( s_X^{ \, C } ) \subseteq S_X ]$: Choose an arbitrary $x_0 \in \textit{cl} \, ( s_X^{ \, C } )$ and an arbitrary $\epsilon_{ x_0 } > 0$. Then either (i) $x_0 \in s_X^{ \, C }$ or (ii) $x_0$ is a limit point of $s_X^{ \, C }$.

\begin{itemize}

\item[(i)]

$x_0 \in s_X^{ \, C }$ implies that $f_X^{ \, C } ( x ) > 0$; thus it must be true that $P ( x_0 - \epsilon_{ x_0 }, x_0 + \epsilon_{ x_0 } ) > 0$, because the only way that $P ( x_0 - \epsilon_{ x_0 }, x_0 + \epsilon_{ x_0 } )$ could be 0 for all $\epsilon_{ x_0 }$ would be for $f_X^{ \, C } ( x_0 ) = 0$.

\item[(ii)]

If instead $x_0$ is a limit point of $s_X^{ \, C }$, then every set $\mathcal{ B } ( x_0, \epsilon_{ x_0 } ) = \{ x_0 - \epsilon_{ x_0 }, x_0 + \epsilon_{ x_0 } \}$ has at least one point $x_{ 00 } \in s_X^{ \, C }$ (i.e., for which $f_X^{ \,  } ( x_{ 00 } ) > 0$) with $x_{ 00 } \ne x_0$; use the argument in (i) again to complete the proof.

\end{itemize}

\end{itemize}

\end{proof}

\textbf{Remark \remark.} All well-regarded (non-measure-theoretic) introductory probability textbooks are written with each family $\{ f_X^{ \, t }, t \in \mathbb{ R } \}$ of PDFs possessed by an absolutely continuous $F_X$ represented by a single PDF. All such textbook distributions that are supported on the entire real line (e.g., Cauchy, Laplace, Logistic, Normal, and $t_\nu$ with $\nu > 1$) are canonical in the sense of \textbf{Definition \ref{d:canonical-pdf-1}}. All such textbook distributions with support on a proper subset of $\mathbb{ R }$ (e.g., Beta, Chi-Square, Exponential, Gamma, Inverse Chi-Square, Inverse Gamma, Log Logistic, Lognormal, continuous Uniform, and Weibull) are either canonical or differ from canonical at most at the one or two points defining the edges of their support; for example, the canonical PDF for the Uniform$[ 0, 1 ]$ distribution is positive only on $( 0, 1 )$, but the support is still $[ 0, 1 ]$
(what to do at the finite sets of edge points is a matter of taste, not necessity). Similar remarks apply to multivariate distributions on $\mathbb{ R }^k$.

\end{quote}

\subsection{The probabilities under independence in the continuous setting} \label{s:continuous-probabilities-1}

To examine what happens to the general definition of independence of two random variables $X$ and $Y$ in the continuous case, we offer the following result. \vspace*{-0.35in}

\begin{quote}

\begin{lemma} \label{l:pdf-factorization-1}

\textbf{\textit{(new)}} Let $F_{ XY } ( x, y )$ be an absolutely continuous CDF on $\mathbb{ R }^2$ with marginal CDFs $F_X ( x )$ and $F_Y ( y )$, possessing canonical PDF versions $f_{ XY }^{ \, C } ( x, y )$, $f_X^{ \, C } ( x )$, and $f_Y^{ \, C } ( y )$. If $X$ and $Y$ are independent in this joint CDF, then for all $( x, y ) \in \mathbb{ R }^2$
\begin{equation} \label{e:independence-pdf-factorization-1}
f_{ XY }^{ \, C } ( x, y ) = f_X^{ \, C } ( x ) \cdot f_Y^{ \, C } ( y ) \, .
\end{equation}
\mbox{ } \vspace*{-0.3in}

\end{lemma}

\begin{proof}

Choose an arbitrary point $( x^*, y^* ) \in \mathbb{ R }^2$; either $F_{ XY }$ is twice differentiable at this point (once in $x$, once in $y$, which would imply differentiability of $F_X$ and $F_Y$) or it's not. If not, both sides of equation (\ref{e:independence-pdf-factorization-1}) are 0 and the result is trivially true; if we \textit{do} have differentiability at $( x^*, y^* )$, by \textbf{Lemma \ref{l:alternative-independence-definition-1}} in Section \ref{s:independence-1}
\begin{eqnarray} \label{e:independence-pdf-factorization-2}
f_{ XY }^{ \, C } ( x, y ) & = & \frac{ \partial^2 }{ \partial x \, \partial y } \, F_{ XY } ( x, y ) = \frac{ \partial^2 }{ \partial x \, \partial y } \bigg[ F_X^{ \, C } ( x ) \cdot F_Y^{ \, C } ( y ) \bigg] \nonumber \\
& = & \bigg[ \frac{ \partial }{ \partial x } F_X^{ \, C } ( x ) \bigg] \cdot \bigg[ \frac{ \partial }{ \partial y } F_Y^{ \, C } ( y ) \bigg] \nonumber \\
& = & f_X^{ \, C } ( x ) \cdot f_Y^{ \, C } ( y ) \, ,
\end{eqnarray}
as desired.

\end{proof}

\end{quote}

\subsection{Our result in the continuous case} \label{s:continuous-result-1}

We can now state and prove our main result in the continuous case. \vspace*{-0.35in}
\begin{quote}

\begin{proposition} \label{p:continuous-proposition-statement-1}

\textbf{\textit{(new)}} Let $F_{ XY } ( x, y )$ be an absolutely continuous CDF on $\mathbb{ R }^2$ with support sets $\{ S_{ XY }, S_X, S_Y \}$ and canonical PDF versions $f_{ XY }^{ \, C } ( x, y )$, $f_X^{ \, C } ( x )$, and $f_Y^{ \, C } ( y )$. If $X$ and $Y$ are independent, then
\begin{equation} \label{e:continuous-support-1}
S_{ XY } = S_X \times S_Y \, .
\end{equation}
\vspace*{-0.55in}

\end{proposition}

\begin{proof}

Choose an arbitrary point $( x^*, y^* ) \in S_{ XY }$. By \textbf{Lemma \ref{l:S-X-works-not-s-X}}, $S_{ XY } = \textit{cl} \, ( s_{ XY }^{ \, C } )$; using \textbf{Lemma \ref{l:pdf-factorization-1}} and \textbf{Fact $\bm{ ( * ) }$} from the proof of \textbf{Proposition \ref{p:discrete-proposition-statement-1}},
\begin{eqnarray} \label{e:continuous-proof-1}
S_{ XY } & = & \textit{cl} \, \left[ \{ ( x^*, y^* ) \! : f_{ XY }^{ \, C } ( x^*, y^* ) > 0 \} \right] \nonumber \\
& = & \textit{cl} \, \left[ \{ ( x^*, y^* ) \! : f_X^{ \, C } ( x^* ) \cdot f_Y^{ \, C } ( y^* ) > 0 \} \right] \nonumber \\
& = & \textit{cl} \, \left[ \{ ( x^*, y^* ) \! : f_X^{ \, C } ( x^* ) > 0 \ \textrm{and} \ f_Y^{ \, C } ( y^* ) > 0 \} \right] \nonumber \\
& = & \textit{cl} \, \left[ \{ x^*\! : f_X^{ \, C } ( x^* ) > 0 \} \times \{ y^* \! : f_Y^{ \, C } ( y^* ) > 0 \} \right] \, .
\end{eqnarray}
Now, by a basic property of the closure operation in topological spaces (specialized to the setting in this paper), in which $\textit{cl} \, ( A \times B ) = \textit{cl} \, ( A ) \times \textit{cl} \, ( B )$ for any subsets $A$ and $B$ of $\mathbb{ R }^1$, we obtain finally that
\begin{eqnarray} \label{e:continuous-proof-2}
S_{ XY } = \textit{cl} \, \left[ \{ x^*\! : f_X^{ \, C } ( x^* ) > 0 \} \right] \times \textit{cl} \, \left[ \{ y^* \! : f_Y^{ \, C } ( y^* ) > 0 \} \right] = S_X \times S_Y \, .
\end{eqnarray}
Since the point $( x^*, y^* )$ was arbitrary, the proof is complete.

\end{proof}

\textbf{Remark \remark.} In parallel with the discrete case, in which \textbf{Proposition \ref{p:discrete-proposition-2}} offered a conditional version of 
\textbf{Proposition \ref{p:discrete-proposition-statement-1}}, and having defined conditional support sets in the continuous case appropriately (details omitted), an analogous conditional version of \textbf{Proposition \ref{p:continuous-proposition-statement-1}} is available, as follows (proof omitted). \vspace*{0.1in}

\begin{proposition} \label{p:continuous-conditional-proposition-1}

\textbf{\textit{(new)}} Let $S_{ ( X \given Y = y ) }$ and $S_{ ( Y \given X = x ) }$ be the conditional support sets under the same conditions as in \textbf{Proposition \ref{p:continuous-proposition-statement-1}}. If $X$ and $Y$ are independent, then
\begin{equation} \label{e:discrete-support-3}
\mathrm{for \ all} \ y \in \mathbb{ R }, \ S_{ ( X \given Y = y ) } = S_X \ \ \ \mathrm{and} \ \ \ \mathrm{for \ all} \ x \in \mathbb{ R }, \ S_{ ( Y \given X = x ) } = S_Y  \, .
\end{equation}

\end{proposition}

\end{quote}

\section{The general case} \label{s:general-case-1}

Now, finally, consider the setting in which $X$ and $Y$ are arbitrary random variables, to which \textbf{Definitions \ref{d:independence-1}--\ref{d:support-1}} and \textbf{Lemma \ref{l:support-1}} apply. \vspace*{-0.35in}

\begin{quote}

\begin{theorem} \label{t:general-theorem-statement-1}

\textbf{\textit{(new)}} Let $X$ and $Y$ be random variables with marginal support sets $S_X$ and $S_Y$, respectively, and with bivariate support set $S_{ XY }$ (in all three instances using the user-friendly version of the definition of support in \textbf{Lemma \ref{l:support-1}} based on neighborhoods). If $X$ and $Y$ are independent, then
\begin{equation} \label{e:general-theorem-1}
S_{ XY } = S_X \times S_Y \, .
\end{equation}

\end{theorem}

\begin{proof}

As in \textbf{Lemma \ref{l:S-X-works-not-s-X}} above, we show that $S_{ XY } = S_X \times S_Y$ in two steps: $S_{ XY } \subseteq ( S_X \times S_Y )$ and $( S_X \times S_Y ) \subseteq S_{ XY }$. 

\begin{itemize}

\item

$[ S_{ XY } \subseteq ( S_X \times S_Y ) ]$: Let $( x, y ) \in S_{ XY }$. To show that $x \in S_X$ and $y \in S_Y$, let $( x - \epsilon_x, x + \epsilon_x )$ and $( y - \epsilon_y, y + \epsilon_y )$ be intervals centered at $x$ and $y$, respectively, with arbitrary $\epsilon_x > 0$ and $\epsilon_y > 0$. The Cartesian product of these two intervals defines a rectangle in $\mathbb{ R }^2$ of the form
\begin{equation} \label{e:rectangle-1}
\mathcal{ R } [ ( x, y ), ( \epsilon_x, \epsilon_y ) ] \triangleq ( x - \epsilon_x, x + \epsilon_x ) \times ( y - \epsilon_y, y + \epsilon_y ) \, ,
\end{equation}
which \textit{contains} the bivariate ball $\mathcal{ B } [ ( x, y ), \epsilon^* ]$ with $\epsilon^* = \min ( \epsilon_x, \epsilon_y ) > 0$. Then
\begin{equation} \label{e:general-theorem-2}
P \big\{ ( X, Y ) \in \mathcal{ R } [ ( x, y ), ( \epsilon_x, \epsilon_y ) ] \big\} \ge P \big\{ ( X, Y ) \in \mathcal{ B } [ ( x, y ), \epsilon^* ] \big\} > 0 \, ,
\end{equation}
in which the second inequality follows from \textbf{Lemma \ref{l:support-1}}. But since $X$ and $Y$ are independent, 
\begin{eqnarray} \label{e:general-theorem-3}
P \big\{ ( X, Y ) \in \mathcal{ R } [ ( x, y ), ( \epsilon_x, \epsilon_y ) ] \big\} & = & P [ X \in ( x - \epsilon_x, x + \epsilon_x ) ] \cdot \nonumber \\
& & \hspace*{0.25in} P [ Y \in ( y - \epsilon_y, y + \epsilon_y ) ] > 0 \, ,
\end{eqnarray}
and now by \textbf{Fact $\bm{ ( * ) }$} in the proof of \textbf{Proposition \ref{p:discrete-proposition-statement-1}} it follows that
\begin{equation} \label{e:general-theorem-4}
P [ X \in ( x - \epsilon_x, x + \epsilon_x ) ] > 0 \ \ \ \mathrm{and} \ \ \ P [ Y \in ( y - \epsilon_y, y + \epsilon_y ) ] > 0 \, ,
\end{equation}
meaning (since $\epsilon_x$ and $\epsilon_y$ are arbitrary positive real numbers) that $x \in S_X$ and $y \in S_Y$, as was to be shown in this step of the proof.

\item

$[ ( S_X \times S_Y ) \subseteq S_{ XY } ]$: Let $( x, y ) \in ( S_X \times S_Y )$, so that $x \in S_X$ and $y \in S_Y$. Construct the rectangle $\mathcal{ R } [ ( x, y ), ( \epsilon_x, \epsilon_y ) ]$ as in the first case (again with arbitrary $\epsilon_x > 0$ and $\epsilon_y > 0$), but this time build the ball $\mathcal{ B } [ ( x, y ), \epsilon^{ ** } ]$ with $\epsilon^{ ** } = \sqrt{ \epsilon_x^2 + \epsilon_y^2 } > 0$ big enough so that the rectangle lies \textit{inside} the ball, yielding $\mathcal{ B } [ ( x, y ), \epsilon^{ ** } ] \supseteq [ ( x - \epsilon_x, x + \epsilon_x ) \times ( y - \epsilon_y, y + \epsilon_y ) ]$. Then
\begin{eqnarray} \label{e:general-theorem-5}
P \big\{ ( X, Y ) \in \mathcal{ B } [ ( x, y ), \epsilon^{ ** } ] \big\} & \ge & P \big\{ ( X, Y ) \in \mathcal{ R } [ ( x, y ), ( \epsilon_x, \epsilon_y ) ] \big\} \nonumber \\
& = & P \big[X \in ( x - \epsilon_x, x + \epsilon_x ) \ \mathrm{and} \nonumber \\
& & \hspace*{0.25in} Y \in ( y - \epsilon_y, y + \epsilon_y ) \big] \, ;
\end{eqnarray}
but by independence (\ref{e:general-theorem-5}) becomes
\begin{eqnarray} \label{e:general-theorem-6}
P \big\{ ( X, Y ) \in \mathcal{ B } [ ( x, y ), \epsilon^{ ** } ] \big\} & \ge & P \big[ X \in ( x - \epsilon_x, x + \epsilon_x ) \big] \cdot \nonumber \\
& & \hspace*{0.25in} P \big[ Y \in ( y - \epsilon_y, y + \epsilon_y ) \big] \, .
\end{eqnarray}
Now, since (a) $\epsilon_x$ and $\epsilon_y$ are arbitrary positive real numbers and (b) by assumption $x \in S_X$ and $y \in S_Y$, it follows that both of $P \big[ X \in ( x - \epsilon_x, x + \epsilon_x ) \big]$ and $P \big[ Y \in ( y - \epsilon_y, y + \epsilon_y ) \big]$ are positive, meaning that $P \big\{ ( X, Y ) \in \mathcal{ B } [ ( x, y ), \epsilon^{ ** } ] \big\} > 0$, i.e., $( x, y ) \in S_{ XY }$ as desired.

\end{itemize}

\end{proof}

\textbf{Remark \remark.} With reference to the \textit{Lebesgue Decomposition Theorem} on the CDF scale in equation (\ref{e:lebesgue-decomposition-1}), the general setting in this Section of the paper corresponds to \textit{any} $\bm{ \pi }$ with $0 \le \pi_i \le 1$ (for $( i = 1, 2, 3 )$) and $\sum_{ i = 1}^3 \pi_i = 1$.

\textbf{Remark \remark.} The generality of this result covers the interesting \textit{mixed} case in which one of the two random variables in a bivariate distribution is discrete and the other one is continuous (for which $\pi_1$ and $\pi_2$ are both in $( 0, 1 )$, summing to 1, and $\pi_3 = 0$). For example, in a simple Bayesian setting we might have $X \sim \textrm{Beta} ( \alpha, \beta )$ (with $\alpha > 0$ and $\beta > 0$) and $( Y \given X ) \sim \textrm{Bernoulli} ( X )$; the joint PMF/PDF would then be $( pf )_{ XY } ( x, y ) = c \, x^{ ( \alpha + y ) - 1 } ( 1 - x )^{ ( \beta + 1 - y ) - 1 }$ for $[ 0 \le x \le 1$ and $y \in \{ 0, 1 \} ]$ and 0 otherwise (with $c > 0$ as a normalizing constant), and the marginal PMF for $Y$ would then be Bernoulli$\left( \frac{ \alpha }{ \alpha + \beta } \right)$.

\end{quote}

\section{Examples} \label{s:examples-1}

We now offer three applications of our basic result.

\begin{itemize}

\item

\textbf{Example \example \ ((Necessary) $<$ (Necessary and Sufficient) with PDFs).} Consider the following bivariate PDF:
\begin{equation} \label{e:addition-not-multiplication}
f_{ XY } ( x, y ) = \left\{ \begin{array}{ccc} c \, [ g ( x ) + h ( y ) ] & \textrm{if} & ( x, y ) \in [ 0, 1 ] \times [ 0, 1 ] \\ 0 & & \textrm{otherwise} \end{array} \right\} \, ,
\end{equation}
in which (a) $g \! : \! [ 0, 1 ] \rightarrow ( 0, \infty )$ and $h \! : \! [ 0, 1 ] \rightarrow ( 0, \infty )$ are both integrable, strictly positive on their common domain $[ 0, 1 ]$, and non-constant and (b) $c$ is chosen to make $f_{ XY }$ integrate to 1; the simplest such example is perhaps the diamond- or kite-shaped PDF $f_{ XY } ( x, y ) = ( x + y )$ on the unit square, as displayed in Figure \ref{f:example-8}. Here, for all $g ( \cdot )$ and $h ( \cdot )$ as specified above, $S_X = S_Y = [ 0, 1 ]$ and $S_{ XY } = S_X \times S_Y$, so the necessary condition in \textbf{Theorem \ref{t:general-theorem-statement-1}} and \textbf{Proposition \ref{p:continuous-proposition-statement-1}} is satisfied, but the \textit{additive} nature of the $f_{ XY }$ in equation (\ref{e:addition-not-multiplication}) ensures that the joint PDF does not factor into the \textit{product} of the marginals; in the specific example in Figure \ref{f:example-8}, $f_X ( x ) = x + \frac{ 1 }{ 2 }$ on $[ 0, 1 ]$, $f_Y ( y ) = y + \frac{ 1 }{ 2 }$ on the same support, and $\left( x + \frac{ 1 }{ 2 }\right) \cdot \left( y + \frac{ 1 }{ 2 }\right) \ne ( x + y )$. Here our result does not offer a short-cut in assessing independence, because (Necessary) $<$ (Necessary and Sufficient).

\begin{figure}[t!]

\centering

\caption{\textit{A perspective plot of the bivariate PDF in \textbf{Example 7} (higher PDF values in yellow)}.}

\label{f:example-8}

\bigskip

\includegraphics[ scale = 0.6 ]{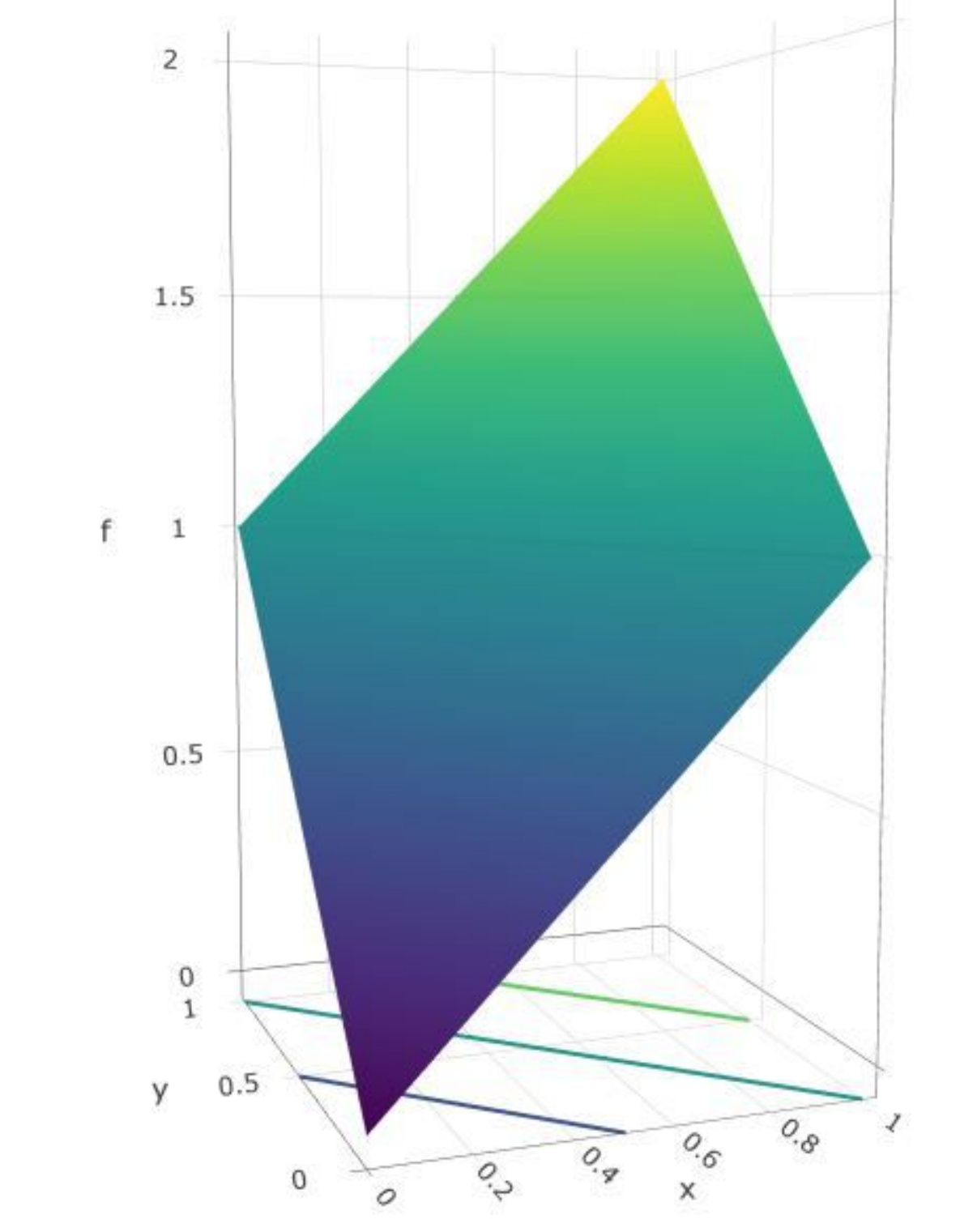}

\end{figure}

\item

\textbf{\textbf{Example \example \ ((Necessary) $<$ (Necessary and Sufficient) with PMFs).}} Imagine making $n = 2$ random draws $( X, Y )$ (with $X$ as the first draw) from a finite population $\mathcal{ P } = \{ v_1, \, \dots, \, v_N \}$, in which $N \ge 3$ is a finite integer and the $v_i$ are real numbers (for $( i = 1, \, \dots, \, N )$); without loss of generality we may take $N = 3$ and $\mathcal{ P } = \{ 4, 5, 7 \}$ for illustration. Consider $P_{ IID } \big[ ( X = 7 ) \ \mathrm{and} \ ( Y = 7 ) \big]$ and $P_{ SRS } \big[ ( X = 7 ) \ \mathrm{and} \ ( Y = 7 ) \big]$, in which IID and SRS are (\textit{independent identically distributed sampling}) and (\textit{simple random sampling}), respectively. Under IID sampling, $S_X = S_Y = \{ 4, 5, 7 \}$ and $S_{ XY }$ is as in the left display in Table \ref{t:iid-srs-1}; under SRS, $S_X$ and $S_Y$ are again both equal to $\{ 4, 5, 7 \}$ and $S_{ XY }$ is summarized in the right display in Table \ref{t:iid-srs-1}. In both cases it's clear that $S_{ XY } = S_X \times S_Y$, so the necessary condition for independence in \textbf{Theorem \ref{t:general-theorem-statement-1}} and \textbf{Proposition \ref{p:discrete-proposition-statement-1}} is met; but $X$ and $Y$ are independent under IID sampling and dependent with SRS, as is clear by inspection of Table \ref{t:iid-srs-1}. Once again, as in \textbf{Example 7}, (Necessary) is weaker than (Necessary and Sufficient).

\begin{table}[t!]

\centering

\caption{\textit{Bivariate sample spaces in \textbf{Example 8}. Entries in left display: $S_{ XY }$ with IID sampling; entries in right display: $S_{ XY }$ with SRS. In both tables, the margins specify $S_X$ (rows) and $S_Y$ (columns); {\small $\bigotimes$} means that the corresponding ordered pair is not possible.}} 

\bigskip

\begin{tabular}{cc|c|c|c|ccc|c|c|c|}

\multicolumn{5}{c}{\textbf{IID}} & & \multicolumn{5}{c}{\textbf{SRS}} \\ \cline{1-5} \cline{7-11}

& \multicolumn{1}{c}{} & \multicolumn{3}{c}{$y$} & & & \multicolumn{1}{c}{} & \multicolumn{3}{c}{$y$} \\

& \multicolumn{1}{c}{} & \multicolumn{1}{c}{4} & \multicolumn{1}{c}{5} & \multicolumn{1}{c}{7} & \multicolumn{1}{c}{} & & \multicolumn{1}{c}{} & \multicolumn{1}{c}{4} & \multicolumn{1}{c}{5} & \multicolumn{1}{c}{7} \\ \cline{3-5} \cline{9-11}

& 4 & $( 4, 4 )$ & $( 4, 5 )$ & $( 4, 7 )$ & & & 4 & $\bigotimes$ & $( 4, 5 )$ & $( 4, 7 )$ \\ \cline{3-5} \cline{9-11}

$x$ & 5 & $( 5, 4 )$ & $( 5, 5 )$ & $( 5, 7 )$ & \mbox{\hspace*{0.6in}} & $x$ & 5 & $( 5, 4 )$ & $\bigotimes$ & $( 5, 7 )$ \\ \cline{3-5} \cline{9-11}

& 7 & $( 7, 4 )$ & $( 7, 5 )$ & $( 7, 7 )$ & & & 7 & $( 7, 4 )$ & $( 7, 5 )$ & $\bigotimes$ \\ \cline{3-5} \cline{9-11}

\end{tabular}

\label{t:iid-srs-1}

\end{table}

\begin{figure}[t!]

\centering

\caption{\textit{Support sets in \textbf{Example 9}. Left panel: bivariate support of $( X_1, X_2 )$ (inside and boundary of green unit square), with $S_{ X_1 }$ and $S_{ X_2 }$ in red; right panel: bivariate support of $( Y_1, Y_2 )$ (boundary and region between green curve and horizontal axis), with $S_{ Y_1 } \times S_{ Y_2 }$ (half-open rectangle) in red}.}

\label{f:example-9}

\begin{tabular}{cc}

\begin{tabular}{c}

\includegraphics[ scale = 0.4 ]{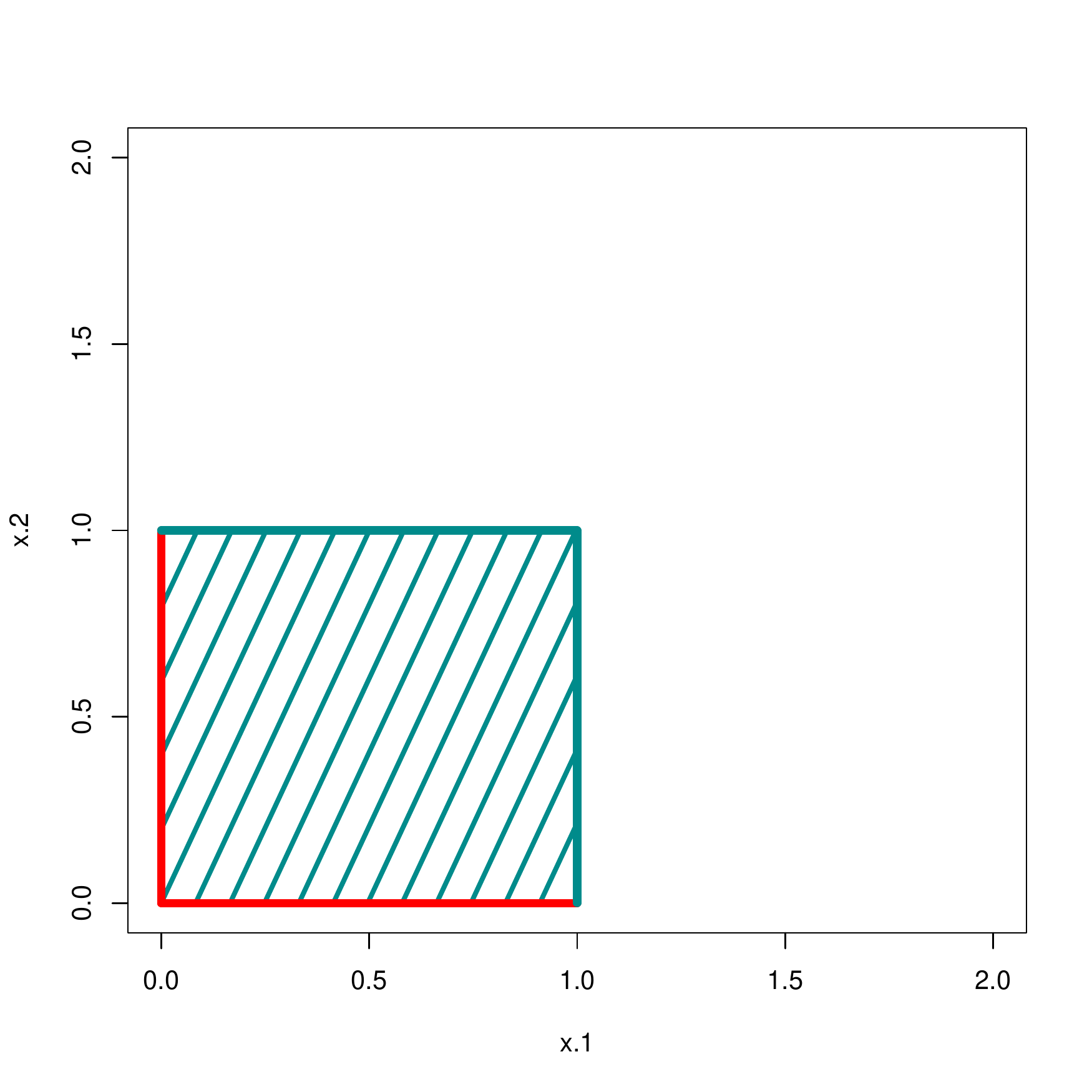} \\

\end{tabular}

&

\begin{tabular}{c}

\includegraphics[ scale = 0.4 ]{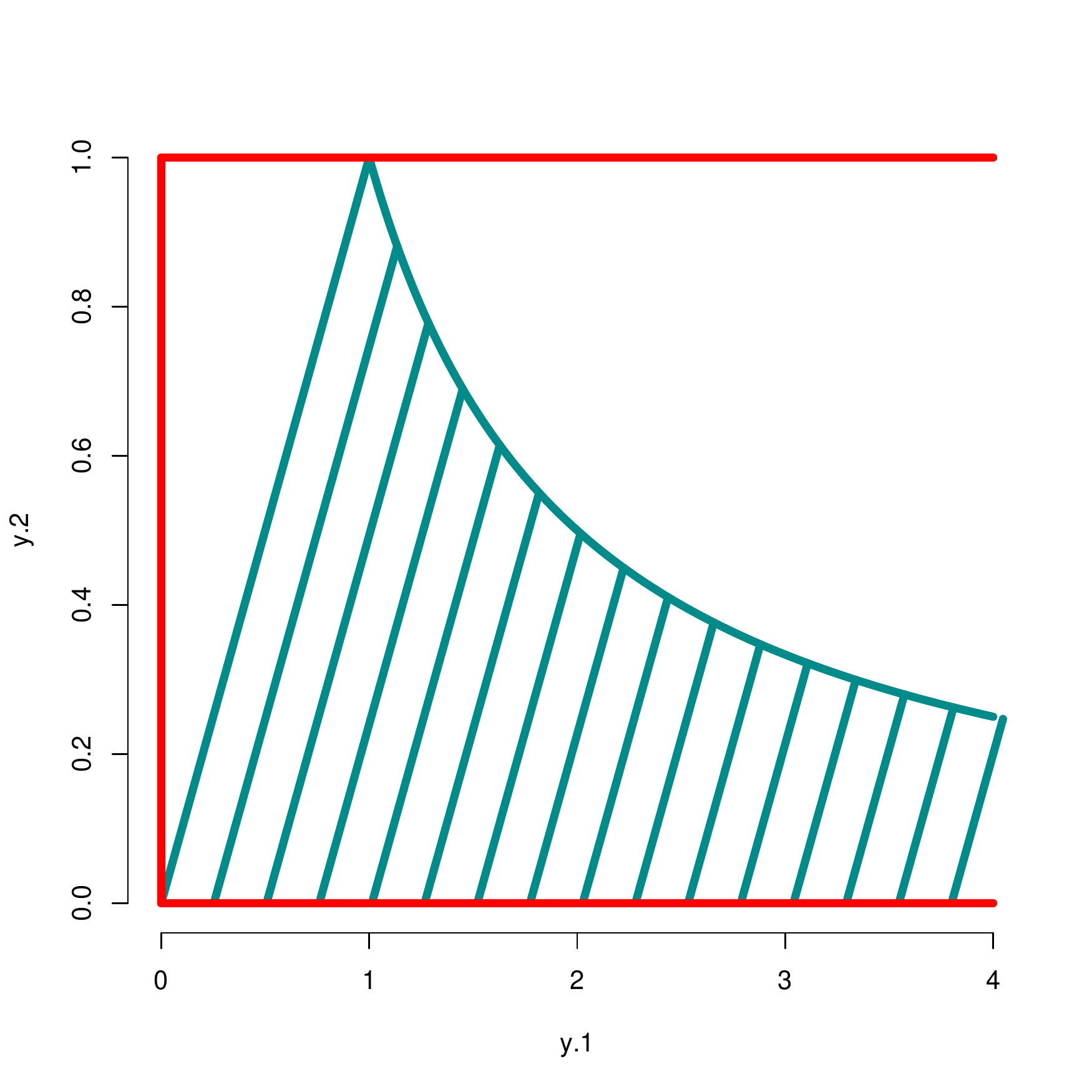}

\end{tabular}

\end{tabular}

\end{figure}

\item

\textbf{Example \example \ \textit{(\cite{degroot-schervish-2012}, p.~183).}} In a departure from the previous notation in the paper, let $( X_1, X_2 )$ have the following joint PDF:
\begin{equation} \label{e:avoid-the-jacobian-1}
f_{ X_1 X_2 } ( x_1, x_2 ) = \left\{ \begin{array}{ccc} 4 \, x_1 \, x_2 & \textrm{if} & ( x_1, x_2 ) \in [ 0, 1 ] \times [ 0, 1 ] \\ 0 & & \textrm{otherwise} \end{array} \right\} \, .
\end{equation}
$X_1$ and $X_2$ are clearly independent in this joint PDF, with marginal PDFs $f_{ X_i } ( x_i ) = 2 \, x_i$ on $S_{ X_1 } = S_{ X_2 } = [ 0, 1 ]$ for $( i = 1, 2 )$ (the left panel of Figure \ref{f:example-9} illustrates the support sets for $( X_1, X_2 )$). Consider the interesting transformation under which $( Y_1, Y_2 ) = \big[ h_1 ( X_1, X_2 ), h_2 ( X_1, X_2 ) \big] \triangleq \left( \frac{ X_1 }{ X_2 }, X_1 \cdot X_2 \right)$; intuitively $Y_1$ and $Y_2$ must be dependent, for example because they share the common multiplicative factor $X_1$ (interestingly, the correlation between $Y_1$ and $Y_2$ is about $-0.13$, when intuition might have suggested a positive association; the reason is the long right tail for large values of $Y_1$ in Figure \ref{f:example-9-perspective-plot}). It's a matter of some tedium to demonstrate this in the usual way, because (a) you need to compute the Jacobian matrix of the inverse transformation to get the joint PDF of $( Y_1, Y_2 )$ and then (b) you have to extract the marginals for the $Y_i$; our method also requires some attention to detail but may arguably be less tedious, as follows. Considering the realized values $( x_1, x_2 )$ and $( y_1, y_2 )$ of $( X_1, X_2 )$ and $( Y_1, Y_2 )$, respectively, a brief calculation reveals that the inverse transformation is given by $\left[ h_1^{ -1 } ( y_1, y_2 ), h_2^{ -1 } ( y_1, y_2 ) \right] = ( x_1, x_2 ) \triangleq \left[ \sqrt{ y_2 \cdot y_1 }, \sqrt{ \frac{ y_2 }{ y_1 } } \right]$. Solving the system of inequalities $\{ x_1 > 0, x_1 < 1, x_2 > 0, x_2 < 1 \}$ in the new coordinates $( y_1, y_2 )$ yields the set $s_{ Y_1 Y_2 }$ in the right panel of Figure \ref{f:example-9} (the oddly-shaped region on and below the green piecewise curve), which is the transformed image of the unit square $s_{ X_1 X_2 }$ in the left panel of that Figure; here $S_{ Y_1 } = [ 0, \infty )$ and $S_{ Y_2 } = [ 0, 1 ]$. Their Cartesian product (the half-open rectangle outlined in red in the right panel) is clearly unequal to $S_{ Y_1 Y_2 }$, i.e., our \textbf{Theorem \ref{t:general-theorem-statement-1}} and \textbf{Proposition \ref{p:continuous-proposition-statement-1}} demonstrate dependence of $Y_1$ and $Y_2$. We conclude this slightly strange example with a perspective plot (Figure \ref{f:example-9-perspective-plot}) of the more-than-slightly-strange joint PDF in the $( y_1, y_2 )$ coordinate system. \vspace*{-0.1in}

\end{itemize}

\begin{figure}[t!]

\centering

\caption{\textit{A perspective plot of the bivariate PDF in \textbf{Example 9} (higher PDF values in yellow)}.}

\label{f:example-9-perspective-plot}

\bigskip

\includegraphics[ scale = 0.6 ]{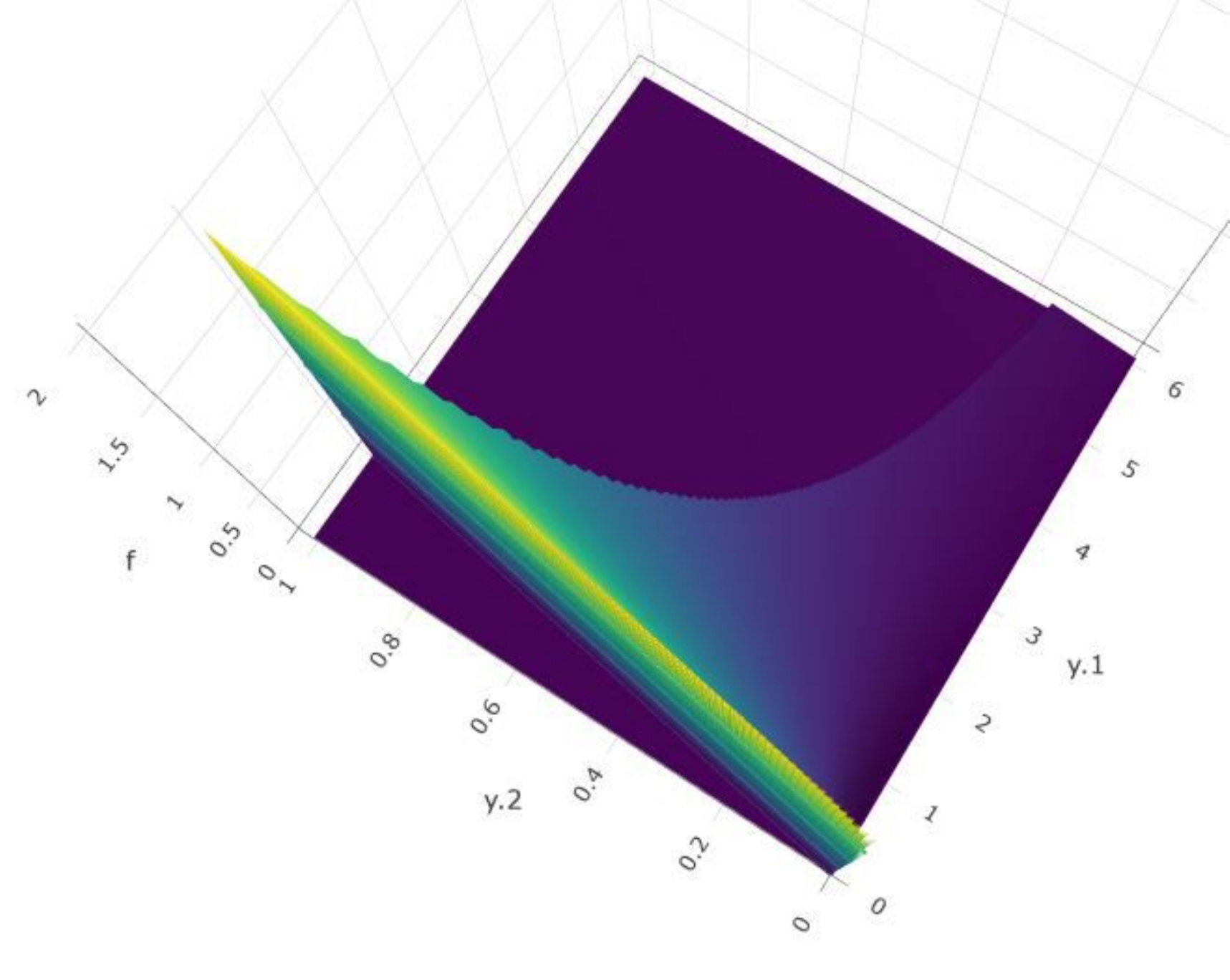}

\end{figure}

\section{Discussion} \label{s:discussion-1}

The environment in which two real-valued random variables $X$ and $Y$ are independent is replete with factorization: under independence, the joint CDF factors into the product of the marginal CDFs, and the same is true on the PMF and (with some care in stating the result) PDF scales. In this paper we add yet another scale with analogous behavior: if $X$ and $Y$ are independent, the joint support set factors into the product of the marginal support sets. The contrapositive of this result offers a simple necessary condition for independence: if the joint support does not correctly factor, the random variables \textit{must} be dependent. This will in some cases ease the burden of proof when exploring independence, as the examples in Section \ref{s:examples-1} illustrate. It has been an interesting journey, crucially involving simple ideas in measure theory and topology, to demonstrate this basic probabilistic finding. \vspace*{-0.1in}

\section*{Appendix}

Here we present a proof sketch that supports our main finding at a high level of abstraction, based on suggestions from Terenin (2021, personal communication) and definitions and theorems from \cite{kallenberg-2021}, abbreviated $\mathbb{ K }$ in what follows; also see \cite{williams-1991} for a deeply and highly usefully intuitive account of the fundamental measure theory and topology needed here. Familiarity with the following topics is assumed in this Appendix: measurable function, measurable space, probability space, product measure, product topology, and topological space.

\begin{itemize}

\item

Start with an arbitrary probability space $( \Omega, \mathcal{ F }, P )$, and consider 
a measurable function $X$ that takes $\Omega$ into a measurable space $( S, \mathcal{ S } )$; $\mathbb{ K }$ calls $X$ a \textit{random element \underline{in}} $S$. By the definition of measurable functions, for any set $B \in \mathcal{ S }$ we can speak meaningfully (a) of the set $X^{ -1 } ( B ) \triangleq \{ \omega \in \Omega \! : X ( \omega ) \in B \}$ and (b) of the derived probabilities $P [ X^{ -1 } ( B ) ] = [ P \circ X^{ -1 } ] ( B )$. As $\mathbb{ K }$ observes, the set function $\mathcal{ L } ( X ) \triangleq P \circ X^{ -1 } \triangleq \mu_X$ is a probability measure on $S$, which may be termed the \textit{distribution} of $X$; $\mathbb{ K }$ differs from ordinary usage in reserving the term \textit{random variable} only for those situations in which $S = \mathbb{ R }$. 

\item

To get two or more independent random elements up and running, $\mathbb{ K }$ establishes the following results. \vspace*{-0.35in}

\begin{quote}

\begin{lemma} \label{l:kallenberg-product-measure-1}

\textbf{(Kallenberg (2021))} Let $( S, \mathcal{ S }, \mu )$ and $( T, \mathcal{ T }, \nu )$ be $\sigma$-finite measure spaces. Then there exists a unique (product) measure $( \mu \otimes \nu )$ on $( S \times T, \mathcal{ S } \otimes \mathcal{ T } )$ such that
\begin{equation} \label{e:product-measure-1}
( \mu \otimes \nu ) ( B \times C ) = [ \mu ( B ) ] \cdot [ \nu ( C ) ] \ \ \ \textrm{for all} \ \ \ B \in \mathcal{ S } \textrm{ and } C \in \mathcal{ T } \, .
\end{equation} 

\end{lemma}

\textbf{Remark \remark.} This result extends with no new ideas to $n$ $\sigma$-finite measure spaces for all finite integers $n \ge 1$.

\begin{lemma} \label{l:kallenberg-independence-1}

\textbf{(Kallenberg (2021))} Let $\{ X_1, \dots, X_n \}$ be random elements with distributions $\{ \mu_1, \dots, \mu_n \}$, respectively, in some measurable spaces $\{ ( S_1, \mathcal{ S }_1 ), \dots,$ $( S_n, \mathcal{ S }_n ) \}$. Then the $X_i$ are independent iff $\bm{ X } = ( X_1, \dots, X_n )$ has distribution $( \mu_1 \otimes \dots \otimes \mu_n )$.

\end{lemma}

\end{quote}

\item

$\mathbb{ K }$'s general definition of support is as follows. \vspace*{-0.35in}

\begin{quote}

\begin{definition}

\textbf{\textit{(Kallenberg (2021))}} For any measure $\mu$ on a topological space $S$, the \textit{support} $S_\mu$ of $\mu$ is the set of points $s \in S$ such that $\mu ( B ) > 0$ for every neighborhood $B$ of $s$.

\end{definition}

\textbf{Remark \remark.} This matches the first result in Billingsley's \textbf{Lemma \ref{l:support-1}} in Section \ref{s:support-sets-1}, when specialized to $\mathbb{ R }^k$.

\end{quote}

\item

We can now state our basic result at this level of abstraction. \vspace*{-0.35in}

\begin{quote}

\begin{theorem} \label{t:our-result-abstractly-1}

\textbf{(new)} Let $\{ X_1, \dots, X_n \}$ be random elements with distributions $\{ \mu_1, \dots, \mu_n \}$, respectively, in some measurable spaces $\bm{ S } = \{ ( S_1, \mathcal{ S }_1 ), \dots, ( S_n, \mathcal{ S }_n ) \}$, in which $\bm{ S }$ is equipped with the product topology. If the $X_i$ are independent, then the support of the product measure $\bm{ \mu } \triangleq ( \mu_1 \otimes \dots \otimes \mu_n )$ must factor:
\begin{equation} \label{e:our-result-abstractly-1}
S_{ \bm{ \mu } } = \left( S_{ \mu_1 } \times ... \times S_{ \mu_n } \right) \, .
\end{equation}

\end{theorem}

\begin{proof}

\textbf{\textit{(sketch)}} The \textit{exact} same proof as in \textbf{Theorem \ref{t:general-theorem-statement-1}} works here with the obvious necessary modifications: 

\begin{itemize}

\item[(1)] 

Show that the two sets in equation (\ref{e:our-result-abstractly-1}) are equal by showing (a) that $S_{ \bm{ \mu } } \subseteq \left( S_{ \mu_1 } \times ... \times S_{ \mu_n } \right)$ and (b) that $\left( S_{ \mu_1 } \times ... \times S_{ \mu_n } \right) \subseteq S_{ \bm{ \mu } }$; 

\item[(2)] 

For (1)(a), pick an arbitrary point $\bm{ p_1 } \in S_{ \bm{ \mu } }$ and build both a (Cartesian product) box and a (neighborhood) ball around $\bm{ p_1 }$ such that the box contains the ball; use independence, $\mathbb{ K }$'s support definition and \textbf{Fact $\bm{ ( * ) }$} from \textbf{Proposition \ref{p:discrete-proposition-statement-1}} to conclude that $\bm{ p_1 } \in \left( S_{ \mu_1 } \times ... \times S_{ \mu_n } \right)$; and

\item[(3)]

For (1)(b), pick an arbitrary point $\bm{ p_2 } \in \left( S_{ \mu_1 } \times ... \times S_{ \mu_n } \right)$ and again build both a (Cartesian product) box and a (neighborhood) ball around $\bm{ p_2 }$, but this time such that the ball contains the box; again use independence, $\mathbb{ K }$'s support definition and \textbf{Fact $\bm{ ( * ) }$} from \textbf{Proposition \ref{p:discrete-proposition-statement-1}} to conclude that $\bm{ p_2 } \in S_{ \bm{ \mu } }$.

\end{itemize}

\end{proof}

\end{quote}

\end{itemize}

\small

\section*{Acknowledgments}

We're grateful to John Kolassa, Alex Terenin, and David Williams for helpful references and comments. Membership on this list does not constitute agreement with the views presented here, nor are any of these people responsible for any errors that may remain.

\bibliography{independence-support}

\end{document}